\documentclass{amsart}

\usepackage{amsmath,amssymb}
\usepackage{amsthm, amsrefs}
\usepackage{latexsym}
\usepackage{indentfirst}
\usepackage{mathrsfs}
\usepackage{placeins}
\usepackage{graphicx}

\numberwithin{equation}{section}

\theoremstyle{plain}
\newtheorem{theorem}{Theorem}[section]
\newtheorem{example}{Example}[section]
\newtheorem{lemma}[theorem]{Lemma}

\theoremstyle{definition}

\theoremstyle{remark}
\newtheorem*{remark}{Remark}

\DeclareMathOperator{\tr}{tr}

\newcommand{\ud}{\,\mathrm{d}}
\newcommand{\RR}{\mathbb{R}}

\newcommand{\TT}{\mathrm{T}}
\newcommand{\LL}{\mathcal{L}}

\newcommand{\Or}{\mathcal{O}}

\newcommand{\bd}[1]{\boldsymbol{#1}}
\newcommand{\wt}[1]{\widetilde{#1}}
\newcommand{\wh}[1]{\widehat{#1}}

\newcommand{\veps}{\varepsilon}

\newcommand{\abs}[1]{\lvert#1\rvert}
\newcommand{\norm}[1]{\lVert#1\rVert}

\newcommand{\I}{\imath}

\newcommand{\nn}{\nonumber}
\newcommand{\dsp}{\displaystyle}

\newcommand{\barint}{\kern4pt \raise3.4pt\hbox{\vrule height.6pt
    width7pt} \kern-11pt \int}

\newcommand{\FGA}{\mathrm{FGA}}
\newcommand{\EFGA}{\mathrm{EFGA}}
\newcommand{\SLFGA}{\mathrm{SLFGA}}

\begin{document}

\title[Frozen Gaussian approximation and Eulerian methods]{Frozen
  Gaussian approximation for general linear strictly hyperbolic
  system: formulation and Eulerian methods}

\author{Jianfeng Lu}
\address{Department of Mathematics \\
  Courant Institute of Mathematical Sciences \\
  New York University \\
  New York, NY 10012 \\
  email: jianfeng@cims.nyu.edu }

\author{Xu Yang}
\address{Department of Mathematics \\
  Courant Institute of Mathematical Sciences \\
  New York University \\
  New York, NY 10012 \\
  email: xuyang@cims.nyu.edu }

\date{October 1st, 2010; Revised: \today}

\thanks{Part of the work was done when both authors visited Peking
  University. We appreciate their hospitality. X.Y. was partially supported
by the DOE grant DE-FG02-03ER25587, the NSF grant DMS-0708026
and the AFOSR grant FA9550-08-1-0433.}

\begin{abstract}
  The frozen Gaussian approximation, proposed in [Lu and Yang,
  \cite{LuYang:CMS}], is an efficient computational tool for high
  frequency wave propagation.  We continue in this paper the
  development of frozen Gaussian approximation. The frozen Gaussian
  approximation is extended to general linear strictly hyperbolic
  systems. Eulerian methods based on frozen Gaussian approximation are
  developed to overcome the divergence problem of Lagrangian
  methods. The proposed Eulerian methods can also be used for the
  Herman-Kluk propagator in quantum mechanics. Numerical examples
  verify the performance of the proposed methods.
\end{abstract}

\maketitle

\section{Introduction}

This is the second of a series of papers on frozen Gaussian
approximation for computing high frequency wave propagation. In the
previous paper \cite{LuYang:CMS}, we proposed frozen Gaussian
approximation for linear scalar wave equation with high frequency
initial condition. It provides a valid approximate solution both in
the presence of caustics and when the solution to wave propagation
spreads. The frozen Gaussian approximation is based on asymptotic
analysis in the phase space, and has better asymptotic accuracy than the
Gaussian beam method \cite{Po:82}. The numerical algorithm based on
frozen Gaussian approximation was proposed in \cite{LuYang:CMS}
within the Lagrangian framework.

In the current paper, we provide an efficient methodology for
computing high frequency wave propagation for general systems with
smooth coefficients. On the one hand, we generalize the frozen
Gaussian approximation to general linear strictly hyperbolic systems;
on the other hand, we develop numerical methods based on the Eulerian
formulation of frozen Gaussian approximation.  The Eulerian methods
solve the problem of divergence of particle trajectories in the
Lagrangian method.

Computation of wave propagation arises from many applications, for
example seismology and electromagnetic radiation, where wave dynamics
are governed by hyperbolic equations. Direct numerical discretization
of hyperbolic system is formidably expensive when waves are highly
oscillatory. In conventional approaches, the mesh size of
discretization has to be comparable to wavelength or even smaller,
while the domain of computation is determined by medium
size. Disparity between the scales of wave length and medium size
requires a huge number of grid points in each dimension. This makes
computation extremely expensive.  To bypass these difficulties of
conventional approaches, numerical methods based on asymptotic
analysis were developed, for example geometric optics and the Gaussian
beam method. These methods are based on asymptotic analysis in the
physical space, and solve an eikonal equation for the phase function
$S$ and a transport equation for the density $\rho$,
\begin{align}\label{eq:eikonal}
  & \partial_t S+H(\bd{x},\nabla_{\bd{x}} S)=0, \\
  \label{eq:transport}
  & \partial_t \rho+\nabla_{\bd{x}}\cdot\bigl(\rho \nabla_{\bd{p}}H
  \bigr)=0,
\end{align}
where $H(\bd{x},\bd{p})$ is the Hamiltonian function. The asymptotic
expansion of geometric optics breaks down at caustics where the
nonlinear eikonal equation \eqref{eq:eikonal} develops
singularities. The Gaussian beam method replaces the real phase function
$S$ in geometric optics with a complex one, which makes asymptotic
solution valid at caustics \cite{Ra:82}. However, as discussed in
\cite{LuYang:CMS}, the construction of the Gaussian beam solution relies
on Taylor expansion around beam center, therefore it loses accuracy
when the beam spreads so that the width becomes large.

Our previous work \cite{LuYang:CMS} made use of fixed-width Gaussian
functions and carried out asymptotic analysis on phase plane to
approximate the solution of high frequency wave propagation. It not
only overcomes the shortcoming of the Gaussian beam method when beams
spread, but also improves asymptotic accuracy. The numerical method
given in \cite{LuYang:CMS} was of Lagrangian type, which may lose
accuracy when particle trajectories are torn far way from each other
after long time propagation. This divergence problem is of course a
typical shortcoming of Lagrangian methods. One natural way to resolve
it is to use Eulerian methods where numerical computation is done on
fixed mesh grids. In the literature of high frequency wave
computation, many Eulerian methods have been developed, for example
wave front methods and moment-based methods reviewed in
\cite{EnRu:03}, level set methods in geometric optics reviewed in
\cite{Ru:07} and the Eulerian Gaussian beam methods \cites{LeQiBu:07,
  LeQi:09, JiWuYa:08, JiWuYa:10, JiWuYaHu:10, JiWuYa:11, LiRa:09,
  LiRa:10}. The underlying idea of these methods is to augment either
geometric optics or the Gaussian beam method by numerical procedures
based on partial differential equations. In this paper, we propose
Eulerian methods which augment frozen Gaussian approximation by
numerical algorithms based on the Liouville equations, which is solved
locally on the phase space. The proposed methods resolve the
divergence problem in the Lagrangian method of frozen Gaussian
approximation. As a byproduct, these Eulerian methods can be also
applied to the Herman-Kluk propagator \cite{HeKl:84} in quantum
mechanics.

The rest of the paper is organized as follows. In
Section~\ref{sec:FGAgen}, we extend frozen Gaussian approximation
(FGA) to general linear strictly hyperbolic systems. In
Section~\ref{sec:HFWP}, we discuss the application of FGA for high
frequency wave propagation, emphasizing on the choices of parameters
and discretization when the characteristic wave frequency is specified
in initial conditions. The Eulerian formulation of frozen Gaussian
approximation (EFGA) is introduced in Section~\ref{sec:EFGA}.  Two
efficient numerical methods are proposed: Eulerian method and
semi-Lagrangian method. Numerical results are shown in
Section~\ref{sec:numerics}.  We conclude with some remarks in
Section~\ref{sec:conclusion}.

\section{Frozen Gaussian approximation for general linear strictly
  hyperbolic systems}\label{sec:FGAgen}

We consider an $M\times M$ linear hyperbolic system in $d$
dimensional space,
\begin{equation}\label{eq:hypersys}
  \partial_t \bd{u} + \sum_{l=1}^d A_l(\bd{x}) \partial_{x_l} \bd{u} = 0,
\end{equation}
where $\bd{u} = (u_1, \ldots, u_M)^{\TT} : \RR^d \to \RR^M$ and
$A_l: \RR^d \to \RR^{M\times M},\, 1 \leq l \leq d$ are smooth
$M\times M$ matrix valued functions in $\bd{x}$. We assume that the
system is \emph{strictly hyperbolic}, \textit{i.e.}, for any $\bd{p}
\in \RR^d \backslash\{0\}$ and any $\bd{q} \in \RR^d$, the matrix
$\sum_{l=1}^d p_l A_l(\bd{q})$ has $M$ \emph{distinguished}
eigenvalues, denoted as $\{ H_m(\bd{q}, \bd{p}) \}_{m=1}^M$. We
denote by $\bd{L}_m(\bd{q}, \bd{p})$ and $\bd{R}_m(\bd{q}, \bd{p})$
the corresponding left and right eigenvectors,
\begin{align}\label{eigen:left}
  & \sum_{l=1}^d p_l \bd{L}_m^{\TT}(\bd{q}, \bd{p}) A_l(\bd{q}) =
  H_m(\bd{q}, \bd{p}) \bd{L}_m^{\TT}(\bd{q}, \bd{p}),\\ \label{eigen:right}
  & \sum_{l=1}^d p_l A_l(\bd{q}) \bd{R}_m(\bd{q}, \bd{p}) =
  H_m(\bd{q}, \bd{p}) \bd{R}_m(\bd{q}, \bd{p}),
\end{align}
with the normalization
\begin{equation*}
  \bd{L}_m^{\TT}(\bd{q}, \bd{p}) \bd{R}_n(\bd{q}, \bd{p}) =
  \delta_{mn},
\end{equation*}
where $\delta_{mn}$ is the Kronecker delta function. As a result of
the smoothness of $A_l$, the eigenvalues $H_m$ and eigenvectors
$\bd{L}_m$, $\bd{R}_m$ depend smoothly on $(\bd{q}, \bd{p})$. The
method as presented requires only minor changes to be extended to
hyperbolic system with eigenvalues of constant multiplicity; we will
not go into details.

\subsection{Formulation}
In frozen Gaussian approximation, to the leading order, the solution
of the system \eqref{eq:hypersys} is approximated by the integral
representation,
\begin{equation}\label{eq:FGAlg}
  \bd{u}^{\FGA}(t, \bd{x}) = \frac{1}{(2\pi \veps)^{3d/2}} \sum_{m=1}^M
  \int \bd{a}_{m}(t, \bd{q}, \bd{p})e^{\I\Phi_m/\veps} v_{m,0}(\bd{y}, \bd{q}, \bd{p})
  \ud \bd{y} \ud \bd{p} \ud \bd{q},
\end{equation}
where $v_{m,0}(\bd{y}, \bd{q}, \bd{p})=\bd{L}_m^{\TT}(\bd{q},
\bd{p}) \bd{u}_0(\bd{y})$ and $\I=\sqrt{-1}$ is the imaginary unit.
Here we denote by $\bd{u}_0$ the initial condition of
\eqref{eq:hypersys}.

In \eqref{eq:FGAlg}, the phase function $\Phi_m$
is given by
\begin{multline}\label{eq:phim}
  \Phi_m(t, \bd{x}, \bd{y}, \bd{q}, \bd{p})
  = S_m(t, \bd{q}, \bd{p}) + \frac{\I}{2} \abs{\bd{x} - \bd{Q}_m}^2
  + \bd{P}_m\cdot(\bd{x} - \bd{Q}_m) \\ + \frac{\I}{2} \abs{\bd{y} - \bd{q}}^2
  - \bd{p}\cdot(\bd{y} - \bd{q}).
\end{multline}
Here $(\bd{Q}_m, \bd{P}_m)$ are viewed as functions of $t, \bd{q}$,
and $\bd{p}$. Given $(\bd{q}, \bd{p})$ as parameters, the evolution
of $(\bd{Q}_m(t, \bd{q}, \bd{p}), \bd{P}_m(t, \bd{q}, \bd{p}))$ is
given by the Hamiltonian flow with Hamiltonian function $H_m$,
\begin{equation}\label{eq:Hflow}
  \begin{cases}
    \dsp\frac{\ud \bd{Q}_m}{\ud t} = \partial_{\bd{P}_m} H_m(\bd{Q}_m, \bd{P}_m),
    \\[1em]
    \dsp\frac{\ud \bd{P}_m}{\ud t} = - \partial_{\bd{Q}_m} H_m(\bd{Q}_m, \bd{P}_m),
  \end{cases}
\end{equation}
with initial conditions
\begin{equation}\label{ini:Hflow}
  \bd{Q}_m(0, \bd{q}, \bd{p}) = \bd{q}, \quad
  \text{and} \quad \bd{P}_m(0, \bd{q}, \bd{p}) = \bd{p}.
\end{equation}
The action function $S_m(t,\bd{q}, \bd{p})$, also viewed as
functions of $(t, \bd{q}, \bd{p})$, satisfies
\begin{equation}\label{eq:S}
  \frac{\ud S_m}{\ud t} = \bd{P}_m\cdot\partial_{\bd{P}_m}
  H_m(\bd{Q}_m, \bd{P}_m) - H_m(\bd{Q}_m, \bd{P}_m),
\end{equation}
with initial condition
\begin{equation}
S_m(t,\bd{q}, \bd{p})=0.
\end{equation}

The amplitude $\bd{a}_m$ is given by $\bd{a}_{m}(t,\bd{q},
\bd{p})=\sigma_m(t,\bd{q}, \bd{p}) \bd{R}_m(\bd{Q}_m, \bd{P}_m)$,
where $\sigma_m$ is determined by evolution equation (after dropping
the subscript $m$ for clarity and simplicity),
\begin{multline}\label{eq:sigma}
  \frac{\ud \sigma}{\ud t} + \sigma \bd{L}^{\TT} (\partial_{\bd{P}} H
  \cdot \partial_{\bd{Q}} \bd{R} - \partial_{\bd{Q}} H
  \cdot \partial_{\bd{P}} \bd{R}) + \sigma
  (\partial_{z_k} \bd{L})^{\TT} \bd{F}_{j} Z_{jk}^{-1}  \\
  + \sigma \partial_{z_n} Q_{j} Z_{kn}^{-1} \bd{L}^{\TT} \bigl(
  - \partial_{Q_{j}} A_k + \frac{\I}{2}
  P_{l} \partial_{Q_{j}} \partial_{Q_{k}} A_l \bigr) \bd{R} =
  0,
\end{multline}
with initial condition
\begin{equation}\label{ini:sigma_la}
\sigma(0,\bd{q}, \bd{p})=2^{d/2}.
\end{equation}
In \eqref{eq:sigma}, we have used Einstein's summation convention
and the short hand notations
\begin{align}\label{eq:F}
  &\bd{F}_{j} = - \Bigl( (A_j - \partial_{P_{j}} H) + \I ( \partial_{Q_{j}} H
  - P_{l} \partial_{Q_{j}} A_l ) \Bigr) \bd{R}, \\
  \label{eq:op_zZ}
  &\partial_{\bd{z}}=\partial_{\bd{q}}-\I\partial_{\bd{p}},
  \qquad
  Z=\partial_{\bd{z}}(\bd{Q}+\I\bd{P}).
\end{align}


\subsection{Asymptotic derivation}
We justify the formulation of frozen Gaussian approximation by
asymptotics. We start with the following ansatz for the solution to
\eqref{eq:hypersys} with the initial datum $\bd{u}_0$,
\begin{multline}\label{eq:ansatz}
  \bd{u}(t, \bd{x}) = \frac{1}{(2\pi \veps)^{3d/2}} \sum_{m=1}^M
  \int (\bd{a}_{m,0}(t, \bd{q}, \bd{p}) \\ + \veps
  \bd{a}_{m,1}(t, \bd{q}, \bd{p})) e^{\I\Phi_m/\veps} v_{m,0}(\bd{y}, \bd{q}, \bd{p})
  \ud \bd{y} \ud \bd{p} \ud \bd{q},
\end{multline}
where $v_{m,0}(\bd{y}, \bd{q}, \bd{p})=\bd{L}_m^{\TT}(\bd{q},
\bd{p}) \bd{u}_0(\bd{y})$, the phase function $\Phi_m$ is given in
\eqref{eq:phim}, and $(\bd{Q}_m, \bd{P}_m)$ follows the Hamiltonian
flow \eqref{eq:Hflow}.

We first state some lemmas that will be used
later. The following lemma is essentially the same as that of Lemma
3.1 in \cite{LuYang:CMS}, and also the standard wave packet
decomposition in disguise (see for example \cite{Fo:89}). Hence the proof
will be omitted.
\begin{lemma}\label{lem:FBI}
  For $\bd{u} \in L^2(\RR^d)$, it holds
  \begin{equation}\label{eq:FBI}
    \bd{u}(\bd{x}) = \frac{1}{(2\pi \veps)^{3d/2}}
    \int_{\RR^{3d}} 2^{d/2}
    e^{\frac{\I}{\veps} \Phi(0,\bd{x}, \bd{y}, \bd{q}, \bd{p})}
    \bd{u}(\bd{y}) \ud \bd{y}\ud \bd{p} \ud \bd{q},
  \end{equation}
where
\begin{equation}\label{eq:phi}
   \Phi(0, \bd{x}, \bd{y}, \bd{q}, \bd{p})
  = \frac{\I}{2} \abs{\bd{x} - \bd{q}}^2
  + \bd{p}\cdot(\bd{x} - \bd{q}) \\ + \frac{\I}{2} \abs{\bd{y} - \bd{q}}^2
  - \bd{p}\cdot(\bd{y} - \bd{q}).
\end{equation}

\end{lemma}

Lemma~\ref{lem:veps1} plays an important role in frozen Gaussian
approximation. A similar observation in the case of the Schr\"odinger
equation was first noted by Kay \cite{Ka:06} in asymptotic derivation
of the Herman-Kluk propagator \cite{HeKl:84} in quantum mechanics. This
observation was made precise in the work of Swart and Rousse
\cite{SwRo:09} for the rigorous analysis of the Herman-Kluk
propagator. Our previous work \cite{LuYang:CMS} extended it to linear
wave equations. It is also true in the current case of general linear
strictly hyperbolic systems.  We omit the subscript $m$ in the
statement and proof of the lemma.
\begin{lemma}\label{lem:veps1}
  For any vector valued function $\bd{b}(\bd{y}, \bd{q}, \bd{p})$ and
  matrix valued function $G(\bd{y}, \bd{q}, \bd{p})$ in Schwartz class
  viewed as functions of $(\bd{y}, \bd{q}, \bd{p})$, we have
  \begin{equation}\label{eq:con1}
    \bd{b}(\bd{y}, \bd{q}, \bd{p}) \cdot (\bd{x} - \bd{Q}) \sim
    - \veps \partial_{z_k} ( b_j Z_{jk}^{-1}),
  \end{equation}
  and
  \begin{equation}\label{eq:con2}
    (\bd{x} - \bd{Q})\cdot G(\bd{y}, \bd{q}, \bd{p}) (\bd{x} - \bd{Q})
    \sim \veps (\partial_{z_n} Q_j) G_{jk} Z_{kn}^{-1} + \veps^2
    \partial_{z_r} \bigl( \partial_{z_n} (G_{jk} Z_{kn}^{-1} u) Z_{jr}^{-1}
    \bigr),
  \end{equation}
  where Einstein's summation convention has been used.

  Moreover, for multi-index $\alpha$ that $\abs{\alpha} \geq 3$,
  \begin{equation}\label{eq:con3}
    (\bd{x} - \bd{Q})^{\alpha} \sim \Or(\veps^{\abs{\alpha}-1}).
  \end{equation}
  Here we use the notation $ f \sim g $ to mean that
  \begin{equation}
    \int_{\RR^{3d}} f e^{\frac{\I}{\veps} \Phi} \ud \bd{y} \ud \bd{p} \ud \bd{q}
    = \int_{\RR^{3d}} g e^{\frac{\I}{\veps} \Phi} \ud \bd{y} \ud \bd{p} \ud
    \bd{q}.
  \end{equation}
\end{lemma}

\begin{proof}
Observe that at $t=0$,
\begin{equation*}
  \partial_{\bd{q}}S-(\partial_{\bd{q}}\bd{Q}) \bd{P}+\bd{p}=0,\qquad
  \partial_{\bd{p}}S-(\partial_{\bd{p}}\bd{Q}) \bd{P}=0.
\end{equation*}
Using \eqref{eq:Hflow} and \eqref{eq:S}, we have
\begin{align*}
\partial_t\bigl(\partial_{\bd{q}}S-(\partial_{\bd{q}}\bd{Q})
\bd{P}+\bd{p}\bigr)&=\partial_{\bd{q}}\left(\partial_t S\right)-\partial_{\bd{q}}\left(\partial_t \bd{Q}
\right)\bd{P}-(\partial_{\bd{q}}\bd{Q})\partial_t \bd{P}\\
&=\partial_{\bd{q}}
(\bd{P}\cdot\partial_{\bd{P}}H-H)
-\bigl(\partial_{\bd{q}}
(\partial_{\bd{P}}H
)\bigr)\bd{P}+(\partial_{\bd{q}}\bd{Q})\partial_{\bd{Q}}H
\\ &=(\partial_{\bd{q}}\bd{P})\partial_{\bd{P}}H-(\partial_{\bd{q}}\bd{Q}
\partial_{\bd{Q}}+\partial_{\bd{q}}\bd{P}\partial_{\bd{P}})H
+(\partial_{\bd{q}}\bd{Q})\partial_{\bd{Q}}H
\\
&=0.
\end{align*}
Analogously we have $\displaystyle \partial_t
\bigl(\partial_{\bd{p}}S-(\partial_{\bd{p}}\bd{Q}) \bd{P}\bigr)=0$. Therefore
for all $t>0$,
\begin{equation}\label{eq:SqSp}\partial_{\bd{q}}S-(\partial_{\bd{q}}\bd{Q})
\bd{P}+\bd{p}=0,\qquad
\partial_{\bd{p}}S-(\partial_{\bd{p}}\bd{Q}) \bd{P}=0.\end{equation}

Then straightforward calculations yield
\begin{align*}
&
\partial_{\bd{q}}\Phi=(\partial_{\bd{q}}\bd{P}-\I\partial_{\bd{q}}\bd{Q})(\bd{x}-\bd{Q})-\I
(\bd{y}-\bd{q}), \\
&
\partial_{\bd{p}}\Phi=(\partial_{\bd{p}}\bd{P}-\I\partial_{\bd{p}}\bd{Q})(\bd{x}-\bd{Q})-
(\bd{y}-\bd{q}),
\end{align*}
which implies that
\begin{equation}\label{eq:dzPhi}
\I\partial_{\bd{z}}\Phi=Z(\bd{x}-\bd{Q}),
\end{equation}
where $\partial_{\bd{z}}$ and $Z$ are defined in \eqref{eq:op_zZ}.
The invertibility of $Z$ follows the same argument in
\cite{LuYang:CMS}*{Lemma~3.2}, hence we omit the details here.

Using \eqref{eq:dzPhi}, one has
\begin{align*}
\int_{\RR^{3d}} \bd{b} \cdot (\bd{x} - \bd{Q})  e^{\frac{\I}{\veps}
\Phi} \ud \bd{y} \ud \bd{p} \ud \bd{q} &=\veps\int_{\RR^{3d}} {b}_j
Z^{-1}_{jk} \left(\frac{\I}{\veps}\partial_{z_k}\Phi\right)
e^{\frac{\I}{\veps} \Phi} \ud \bd{y} \ud \bd{p} \ud \bd{q} \\
&=-\veps \int_{\RR^{3d}} \Bigl(\partial_{z_k}\big({b}_j Z^{-1}_{jk} \big)\Bigr)
e^{\frac{\I}{\veps} \Phi} \ud \bd{y} \ud \bd{p} \ud \bd{q},
\end{align*}
where the last equality is obtained from integration by parts.
This proves \eqref{eq:con1}.

Making use of \eqref{eq:con1} twice produces \eqref{eq:con2}
\begin{align*}
(\bd{x} - \bd{Q})\cdot G (\bd{x} - \bd{Q})
    & = (x-Q)_jG_{jk}(x-Q)_k \\
    & \sim -\veps \partial_{z_n}\bigl((x-Q)_jG_{jk}Z^{-1}_{kn}
    \bigr) \\
    & = \veps (\partial_{z_n} Q_j) G_{jk} Z_{kn}^{-1}
    -\veps(x-Q)_j\partial_{z_n}(G_{jk}Z^{-1}_{kn})
    \\
    & \sim \veps (\partial_{z_n} Q_j) G_{jk} Z_{kn}^{-1}   + \veps^2
    \partial_{z_r} \bigl( \partial_{z_n} (G_{jk} Z_{kn}^{-1} ) Z_{jr}^{-1} \bigr).
\end{align*}

By induction it is easy to see that \eqref{eq:con3} is true.

\end{proof}

\subsubsection{Initial value decomposition}

We first check \eqref{eq:FGAlg} gives the right initial value at time
$t = 0$. Obviously it means to take, in \eqref{eq:ansatz},
\begin{equation*}
\bd{a}_{m,0}(0, \bd{q}, \bd{p}) = 2^{d/2} \bd{R}_m(\bd{q}, \bd{p})
\quad\hbox{and}\quad \bd{a}_{m,1}(0, \bd{q}, \bd{p}) = 0.
\end{equation*}
We then have
\begin{equation*}
  \bd{u}(0, \bd{x}) = \frac{1}{(2\pi \veps)^{3d/2}}
  \sum_{m=1}^M \int 2^{d/2} \bd{R}_m(\bd{q}, \bd{p})
  \bd{L}_m^{\TT}(\bd{q}, \bd{p}) e^{i\Phi(0, \bd{x}, \bd{y}, \bd{q}, \bd{p})/\veps}
  \bd{u}_0(\bd{y}) \ud\bd{y}\ud\bd{p}\ud\bd{q},
\end{equation*} where $\Phi(0, \bd{x}, \bd{y}, \bd{q}, \bd{p})$ is given in \eqref{eq:phi}.

By the normalization of $\bd{L}_m,\ \bd{R}_m,\ m = 1, \ldots, M$, we
have
\begin{equation*}
  \sum_{m=1}^M \bd{R}_m(\bd{q}, \bd{p}) \bd{L}_m^{\TT}(\bd{q}, \bd{p})
  = I_M,
\end{equation*}
where $I_M$ is the $M \times M$ identity matrix. Hence
\begin{equation*}
  \bd{u}(0, \bd{x}) = \frac{1}{(2\pi \veps)^{3d/2}}
  \int 2^{d/2} e^{i\Phi(0, \bd{x}, \bd{y}, \bd{q}, \bd{p})/\veps}
  \bd{u}_0(\bd{y}) \ud\bd{y}\ud\bd{p}\ud\bd{q} = \bd{u}_0(\bd{x}).
\end{equation*}
The last equality follows from Lemma~\ref{lem:FBI}.  Therefore
the initial condition is reproduced by \eqref{eq:FGAlg} at $t=0$.

\subsubsection{Evolution equation}

We derive the evolution equation \eqref{eq:sigma} for $\sigma_m$ in
this subsection. Since the system under consideration is linear, we
only need to consider one branch. For ease of notation, we suppress
the subscript $m$ in this section.

Taking derivatives of $\Phi$ with respect to $t$ and $\bd{x}$ produces
\begin{equation*}
  \partial_t \Phi = \partial_t S - \bd{P} \cdot \partial_t\bd{Q}
  + (\bd{x} - \bd{Q}) \cdot (\partial_t\bd{P} - \I \partial_t\bd{Q}),
\end{equation*}
and
\begin{equation*}
\partial_{x_l} \Phi = \I (x_l - Q_l) + P_l.
\end{equation*}
Therefore the derivatives of ansatz \eqref{eq:ansatz} can be
calculated as
\begin{equation*}
  \begin{aligned}
    \partial_t \bd{u} & = \int \Bigl( \partial_t \bd{a}_0 +
    \veps \partial_t \bd{a}_1 + \frac{\I}{\veps} \partial_t \Phi
    (\bd{a}_0 + \veps \bd{a}_1)\Bigr) e^{\I \Phi/\veps}
    {v}_0(\bd{y}, \bd{q}, \bd{p})
    \ud\bd{y}\ud\bd{p}\ud\bd{q} \\
    & = \int \Bigl( \frac{\I}{\veps}(\partial_t S - \bd{P}\cdot
    \partial_t\bd{Q}) \bd{a}_0 + \bigl( \partial_t \bd{a}_0 + \I(\partial_t S
    - \bd{P}\cdot \partial_t\bd{Q})\bd{a}_1 \\
    & \hspace{6em} + \frac{\I}{\veps} (\bd{x} -
    \bd{Q}) \cdot (\partial_t\bd{P} - \I \partial_t\bd{Q}) \bd{a}_0 \bigr) \Bigr)
    e^{\I \Phi/\veps} {v}_0(\bd{y}, \bd{q}, \bd{p})
    \ud\bd{y}\ud\bd{p}\ud\bd{q} + \Or(\veps),
  \end{aligned}
\end{equation*}
and
\begin{equation*}
  \begin{aligned}
    \partial_{x_l} \bd{u} & = \int (\bd{a}_0 + \veps \bd{a}_1)
    \frac{\I}{\veps} \partial_{x_l} \Phi e^{\I \Phi/ \veps}
    {v}_0(\bd{y}, \bd{q}, \bd{p})
    \ud\bd{y}\ud\bd{p}\ud\bd{q} \\
    & = \int \Bigl( \frac{\I}{\veps} P_l \bd{a}_0 + \bigl( -
    \frac{1}{\veps} (x_l - Q_l)\bd{a}_0 + \I P_l \bd{a}_1 \bigr)
    \Bigr) e^{\I \Phi/ \veps} {v}_0(\bd{y}, \bd{q}, \bd{p})
     \ud\bd{y}\ud\bd{p}\ud\bd{q} + \Or(\veps)
  \end{aligned}
\end{equation*}
Taylor expansion of $A_l(\bd{x})$ around $A_l(\bd{Q})$ gives
\begin{equation*}
  A_l(\bd{x}) = A_l(\bd{Q}) + (x_j - Q_j) \partial_{Q_j} A_l(\bd{Q})
  + \frac{1}{2} (x_j - Q_j)(x_k - Q_k) \partial_{Q_j}\partial_{Q_k}
  A_l(\bd{Q}) + \Or(\bd{x} - \bd{Q})^3.
\end{equation*}
Substituting the above expressions into equation \eqref{eq:hypersys} and
matching orders in $\veps$ yield the leading order equation,
\begin{equation}
  \int \Bigl( \partial_t S
  - \bd{P}\cdot \partial_t\bd{Q} + \sum_{l=1}^d P_l A_l \Bigr)
  \bd{a}_0 e^{\I \Phi/\veps} {v}_0(\bd{y}, \bd{q}, \bd{p})
   \ud \bd{y} \ud \bd{p} \ud \bd{q} = 0.
\end{equation}
Define the action function $S$ to satisfy
\begin{equation}
  \partial_t S - \bd{P}\cdot\partial_t\bd{Q} = - H(\bd{Q}, \bd{P}),
\end{equation}
or equivalently
\begin{equation}
  S(t, \bd{q}, \bd{p}) = \int_0^t \bd{P}\cdot
  \partial_t \bd{Q} - H(\bd{Q}, \bd{P}) \ud s,
\end{equation}
where $\bd{P}$, $\bd{Q}$ and $\partial_t \bd{Q}$ in the integrand
are evaluated at $(s, \bd{q}, \bd{p})$. If we take
\begin{equation}
  \bd{a}_0(t, \bd{q}, \bd{p}) = \sigma(t, \bd{q}, \bd{p})
  \bd{R}(\bd{Q}, \bd{P}),
\end{equation}
then by the definition of $\bd{R}(\bd{Q}, \bd{P})$ in
\eqref{eigen:right},
\begin{equation*}
  \Bigl( \partial_t S
  - \bd{P}\cdot \partial_t\bd{Q} + \sum_{l=1}^d P_l A_l \Bigr) \bd{R}(\bd{Q}, \bd{P})
  = \Bigl(\partial_t S - \bd{P}\cdot \partial_t\bd{Q} + H(\bd{Q}, \bd{P})\Bigr)
  \bd{R}(\bd{Q}, \bd{P}) = 0.
\end{equation*}

To determine $\sigma$, we investigate the next order equation,
\begin{equation}
  \begin{aligned}
    \int \biggl( & \I \Bigl(\partial_t S - \bd{P} \cdot \partial_t
    \bd{Q} + P_l A_l \Bigr) \bd{a}_1  + \partial_t \bd{a}_0 \\
    & + \frac{1}{\veps} (x_j - Q_j) \Bigl( \I(\partial_t P_j -
    \I \partial_t Q_j)\bd{a}_0 - A_j \bd{a}_0 +
    \I  P_l \partial_{Q_j} A_l \bd{a}_0 \Bigr) \\
    & + \frac{1}{\veps} (x_j - Q_j)(x_k - Q_k) \Bigl( - \partial_{Q_j}
    A_k \bd{a}_0 + \frac{\I}{2} P_l \partial_{Q_j} \partial_{Q_k} A_l
    \bd{a}_0 \Bigr) \biggr) \\ &\hspace{16em}\times e^{i\Phi/\veps} v_0
    \ud\bd{y}\ud\bd{p}\ud\bd{q}=0,
  \end{aligned}
\end{equation}
where we interpret the terms quadratic in $\bd{x} - \bd{Q}$ in the
above expression by only keeping the $\Or(\veps)$ term arising from
Lemma~\ref{lem:veps1}.

Solvability condition for $\bd{a}_1$ and Lemma~\ref{lem:veps1}
give the equation of $\bd{a}_0$,
\begin{equation}\label{eq:EvoEq1}
  \begin{aligned}
    \bd{L}(\bd{Q}, \bd{P})^{\TT} \biggl( & \partial_t \bd{a}_0
    {v}_0(\bd{y}, \bd{q}, \bd{p}) \\
    & - \partial_{z_k} \Bigl( \bigl( ( \I \partial_t P_j + \partial_t
    Q_j ) \bd{a}_0 - A_j \bd{a}_0 + \I P_l \partial_{Q_j} A_l \bd{a}_0
    \bigr) Z_{jk}^{-1} {v}_0(\bd{y}, \bd{q}, \bd{p})
    \Bigr) \\
    & + \partial_{z_n} Q_j \bigl( - \partial_{Q_j} A_k \bd{a}_0 +
    \frac{\I}{2} P_l \partial_{Q_j} \partial_{Q_k} A_l \bd{a}_0 \bigr)
    Z_{kn}^{-1} {v}_0(\bd{y}, \bd{q}, \bd{p})
    \biggr) = 0.
  \end{aligned}
\end{equation}
We next expand and simplify the above equation. For the first term,
easy calculations yield
\begin{equation*}
  \bd{L}(\bd{Q}, \bd{P})^{\TT} \partial_t \bd{a}_0
  = \partial_t \sigma + \sigma \bd{L}^{\TT}(\partial_{\bd{P}} H
  \cdot \partial_{\bd{Q}} \bd{R} - \partial_{\bd{Q}} H \cdot
  \partial_{\bd{P}} \bd{R}).
\end{equation*}
To simplify the second term in \eqref{eq:EvoEq1}, we notice that
by the definition \eqref{eigen:right},
\begin{equation*}
  \sum_{l=1}^d P_l A_l(\bd{Q}) \bd{R}(\bd{Q}, \bd{P}) =
  H(\bd{Q}, \bd{P}) \bd{R}(\bd{Q}, \bd{P}).
\end{equation*}
Differentiating the above equation with respect to $\bd{P}$ and $\bd{Q}$ gives
\begin{equation}
\begin{aligned}
  & A_l(\bd{Q}) \bd{R}(\bd{Q}, \bd{P}) + P_j
  A_j(\bd{Q}) \partial_{P_l} \bd{R}(\bd{Q}, \bd{P}) \\
  & \hspace{6em} = \partial_{P_l} H(\bd{Q}, \bd{P}) \bd{R}(\bd{Q},
  \bd{P}) + H(\bd{Q}, \bd{P}) \partial_{P_l} \bd{R}(\bd{Q},
  \bd{P}), \nn
\end{aligned}
\end{equation}
and
\begin{equation}
\begin{aligned}
  & P_j \partial_{Q_l} A_j(\bd{Q}) \bd{R}(\bd{Q}, \bd{P}) + P_j
  A_j(\bd{Q}) \partial_{Q_l}
  \bd{R}(\bd{Q}, \bd{P}) \\
  & \hspace{6em} = \partial_{Q_l} H(\bd{Q}, \bd{P}) \bd{R}(\bd{Q},
  \bd{P}) + H(\bd{Q}, \bd{P}) \partial_{Q_l} \bd{R}(\bd{Q},
  \bd{P}). \nn
\end{aligned}
\end{equation}
Taking inner product with $\bd{L}(\bd{Q}, \bd{P})$ on the left
produces
\begin{align}
  \label{eq:derivP} & \bd{L}^{\TT}(\bd{Q}, \bd{P}) \bigl( A_l(\bd{Q})
  - \partial_{P_l} H(\bd{Q}, \bd{P}) \bigr) \bd{R}(\bd{Q}, \bd{P}) = 0, \\
  \label{eq:derivQ} & \bd{L}^{\TT}(\bd{Q}, \bd{P}) \bigl(
  P_j \partial_{Q_l} A_j(\bd{Q}) - \partial_{Q_l} H(\bd{Q}, \bd{P})
  \bigr) \bd{R}(\bd{Q}, \bd{P}) = 0.
\end{align}
Recall the short hand notation
\begin{equation*}
  \begin{aligned}
    \bd{F}_j & = (\I \partial_t P_j + \partial_t Q_j) \bd{R}
    - A_j \bd{R} + \I P_l \partial_{Q_j} A_l \bd{R} \\
    & = - \Bigl( (A_j - \partial_{P_j} H) + \I ( \partial_{Q_j} H
    - P_l \partial_{Q_j} A_l ) \Bigr) \bd{R}.
  \end{aligned}
\end{equation*}
Using \eqref{eq:derivP} and \eqref{eq:derivQ}, it is clear that for
any $j = 1, \ldots, d$,
\begin{equation*}
  \bd{L}^{\TT}(\bd{Q}, \bd{P}) \bd{F}_j = 0.
\end{equation*}
Hence,
\begin{equation*}
  \begin{aligned}
    \bd{L}^{\TT}(\bd{Q}, \bd{P}) \partial_{z_k} (\sigma \bd{F}_j
    Z_{jk}^{-1} {v}_0(\bd{y}, \bd{q}, \bd{p})) & = \sigma
    \bd{L}^{\TT}(\bd{Q}, \bd{P}) \partial_{z_k} (\bd{F}_j) Z_{jk}^{-1}
    {v}_0(\bd{y}, \bd{q}, \bd{p}) \\
    & \qquad + \bd{L}^{\TT}(\bd{Q}, \bd{P}) \bd{F}_j \partial_{z_k}
    \bigl(
    \sigma  Z_{jk}^{-1} {v}_0(\bd{y}, \bd{q}, \bd{p}) \bigr) \\
    & = - \sigma ( \partial_{z_k} \bd{L})^{\TT}(\bd{Q}, \bd{P}) \bd{F}_j
    Z_{jk}^{-1} {v}_0(\bd{y}, \bd{q}, \bd{p}).
  \end{aligned}
\end{equation*}
Therefore, \eqref{eq:EvoEq1} can be rewritten as
\begin{multline*}
  \partial_t \sigma + \sigma \bd{L}^{\TT} (\partial_{\bd{P}} H
  \cdot \partial_{\bd{Q}} \bd{R} - \partial_{\bd{Q}} H
  \cdot \partial_{\bd{P}} \bd{R}) + \sigma
  (\partial_{z_k} \bd{L})^{\TT} \bd{F}_j Z_{jk}^{-1}  \\
  + \sigma \partial_{z_n} Q_j Z_{kn}^{-1} \bd{L}^{\TT} \bigl(
  - \partial_{Q_j} A_k + \frac{\I}{2}
  P_l \partial_{Q_j} \partial_{Q_k} A_l \bigr) \bd{R} = 0,
\end{multline*}
which is just the evolution equation \eqref{eq:sigma}.

\subsection{Examples}\label{sec:WaveExa}

We apply the general results to some specific systems. For a given
system, once the eigenvalues and eigenfunctions are determined, it is
straightforward to obtain the initial value decomposition and
evolution equation for $\sigma$. We illustrate this by two
examples.

\subsubsection{Scalar wave equation in one dimension}

Consider the 1D scalar wave equation
\begin{equation}\label{eq:wave}
  \partial_t^2 u - c^2(x) \partial_x^2 u = 0,
\end{equation}
where $c(x) > 0$ is the (local) wave speed. Define $r = \partial_t u$ and
$s = \partial_x u$, and transform \eqref{eq:wave} into the system
\begin{equation}\label{eq:1Dwavesys}
  \begin{cases}
    \partial_t r - c(x)^2 \partial_x s = 0, \\
    \partial_t s - \partial_x r = 0.
  \end{cases}
\end{equation}
It can be rewritten as
\begin{equation*}
  \partial_t
  \begin{pmatrix}
    r \\
    s
  \end{pmatrix}
  +
  A
  \partial_x
  \begin{pmatrix}
    r \\
    s
  \end{pmatrix}
  = 0,
\end{equation*}
where $A$ is given by
\begin{equation*}
  A =
  \begin{pmatrix}
    0 & - c(x)^2 \\
    - 1 & 0
  \end{pmatrix}.
\end{equation*}
The eigenvalues of $A$ are given by
\begin{equation*}
  H_{\pm}(q, p) = \pm c(q) \abs{p}.
\end{equation*}
Hence the $2 \times 2$ system \eqref{eq:1Dwavesys} is strictly hyperbolic.
The corresponding right and left eigenvectors are
\begin{equation*}
  \bd{R}_{\pm}(q, p) =
  \begin{pmatrix}
    c(q) \abs{p} \\
    \mp p
  \end{pmatrix}, \quad \bd{L}_{\pm}(q, p) = \frac{1}{2}
  \begin{pmatrix}
    1 / (\abs{p} c(q)) \\
    \mp 1/p
  \end{pmatrix}.
\end{equation*}


By \eqref{eq:sigma}, the evolution equations are given by
\begin{equation*}
  \partial_t \sigma_{\pm} = \pm \frac{\sigma_{\pm}}{2}
  \frac{P_{\pm}}{\abs{P_{\pm}}} c'(Q_{\pm}) \pm \frac{\sigma_{\pm}}{2}
  Z_{\pm}^{-1} \partial_z Q_{\pm}
  \Bigl( 2\frac{P_{\pm}}{\abs{P_{\pm}}} c'(Q_{\pm}) - \I \abs{P_{\pm}}
  c''(Q_{\pm})  \Bigr).
\end{equation*}
This agrees with the evolution equations given in
\cite{LuYang:CMS} in one dimensional case, where the amplitude was
denoted as $a$ instead of $\sigma$.

In \eqref{eq:FGAlg}, the initial value decomposition is taken as
\begin{equation*}
  v_{0, \pm}(y, q, p) = \frac{1}{2} \Bigl( \frac{1}{\abs{p} c(q)}
  \partial_t u(0, y) \mp \frac{1}{p}\partial_y u(0, y)\Bigr).
\end{equation*}
If the initial value to the wave equation takes the WKB form,
\textit{i.e.},
\begin{equation*}
  \begin{cases}
    u_0(x) = A_0(x) e^{\frac{\I}{\veps} S_0(x)}, \\
    \partial_t u_0(x) = \frac{1}{\veps} B_0(x) e^{\frac{\I}{\veps}
      S_0(x)},
  \end{cases}
\end{equation*}
then
\begin{equation*}
  v_{0, \pm}(y, q, p) = \frac{1}{2\veps} \Bigl( \frac{B_0(y)}{\abs{p} c(q)}
  e^{\I S_0(y) /\veps} \mp \frac{1}{p} \big(\I A_0(y) S_0'(y) + \veps A_0'(y)\big)
  e^{\I S_0(y) /\veps} \Bigr).
\end{equation*}

\begin{remark}
  The choices of $\bd{R}$ and $\bd{L}$ are not unique. The above
  choice is made in order to match the results in
  \cite{LuYang:CMS}. If different normalization is chosen for
  $\bd{R}$, the results of initial value decomposition and evolution
  equations can be different.
\end{remark}

\subsubsection{Acoustic wave equation in two dimension}

We next consider the acoustic wave equation in two dimension,
\begin{equation}\label{eq:acoustic}
  \begin{cases}
    \partial_t \bd{V} +  \nabla \Pi = 0, \\
    \partial_t \Pi + c^2(\bd{x}) \nabla\cdot \bd{V} = 0,
  \end{cases}
\end{equation}
where $\bd{V}$ is velocity and $\Pi$ is pressure.  Define $\bd{u} = (V_1,
V_2, \Pi)^{\TT}$, and we can rewrite \eqref{eq:acoustic} as a $3\times 3$ linear
hyperbolic system,
\begin{equation*}
  \partial_t \bd{u} + A_1(\bd{x}) \partial_{x_1} \bd{u}
  + A_2(\bd{x}) \partial_{x_2} \bd{u} = 0,
\end{equation*}
where
\begin{equation*}
  A_1 =
  \begin{pmatrix}
    0 & 0 & 1 \\
    0 & 0 & 0 \\
    c(\bd{x})^2 & 0 & 0
  \end{pmatrix},
  \quad
  A_2 =
  \begin{pmatrix}
    0 & 0 & 0 \\
    0 & 0 & 1 \\
    0 & c(\bd{x})^2 & 0
  \end{pmatrix}.
\end{equation*}

Then the eigenfunctions in \eqref{eigen:left}-\eqref{eigen:right} are given by
\begin{equation*}
  H_1(\bd{q}, \bd{p}) = 0, \quad
  H_{\pm}(\bd{q}, \bd{p}) = \pm c(\bd{q}) \abs{\bd{p}},
\end{equation*}
where we have used the subscripts $\pm$ instead of number
subscripts. This implies the system is strictly hyperbolic, and the
corresponding eigenvectors are
\begin{equation*}
  \begin{aligned}
    & \bd{R}_1(\bd{q}, \bd{p}) =
  \begin{pmatrix}
    p_2  \\
    - p_1 \\
    0
  \end{pmatrix}, \quad \bd{R}_{\pm}(\bd{q}, \bd{p}) =
  \begin{pmatrix}
    \pm p_1  \\
    \pm p_2  \\
    c(\bd{q}) \abs{\bd{p}}
  \end{pmatrix}, \\
  & \bd{L}_1(\bd{q}, \bd{p}) = \frac{1}{\abs{\bd{p}}^2}
  \begin{pmatrix}
    p_2 \\
    - p_1 \\
    0
  \end{pmatrix}, \quad \bd{L}_{\pm}(\bd{q}, \bd{p}) = \frac{1}{2}
  \begin{pmatrix}
    \pm p_1/\abs{\bd{p}}^2 \\
    \pm p_2 /\abs{\bd{p}}^2 \\
    1/ (c(\bd{q}) \abs{\bd{p}} )
  \end{pmatrix}.
  \end{aligned}
\end{equation*}


In \eqref{eq:FGAlg}, the initial value decomposition is taken as
\begin{equation*}
  v_{0,1}(\bd{y}, \bd{q}, \bd{p}) = \frac{1}{\abs{\bd{p}}^2}
  \big(p_2 V_1^0(\bd{y}) - p_1 V_2^0(\bd{y})\big),
\end{equation*}
and
\begin{equation*}
  v_{0, \pm}(\bd{y}, \bd{q}, \bd{p}) = \frac{1}{2\abs{\bd{p}}^2}
  \Bigl(\pm p_1 V_1^0(\bd{y}) \pm p_2 V_2^0(\bd{y}) + \frac{\abs{\bd{p}}}{c(\bd{q})}
  \Pi^0(\bd{y})\Bigr),
\end{equation*}
where $\bd{u}_0 = (V_1^0, V_2^0, \Pi^0)^{\TT}$ is the initial
condition.

The evolution equation \eqref{eq:sigma} of $\sigma$ can be
simplified as, after straightforward but lengthy calculations,
\begin{equation*}
  \frac{\ud \sigma_1}{\ud t} = 0,
\end{equation*}
and
\begin{equation*}
  \begin{aligned}
    \frac{\ud \sigma_{\pm}}{\ud t} = & \pm\frac{\sigma_{\pm}}{2}
    \Bigl( \frac{\bd{P}_{\pm}}{\abs{\bd{P}_{\pm}}}\cdot
    \partial_{\bd{Q}_{\pm}} c - \frac{\I c}{\abs{\bd{P}_{\pm}}} \Bigr) \\
    &\pm \frac{\sigma_{\pm}}{2} \tr\biggl(
    Z_{\pm}^{-1} \partial_{\bd{z}} \bd{Q}_{\pm} \Bigl( 2
    \frac{\bd{P}_{\pm}}{\abs{\bd{P}_{\pm}}}
    \otimes \partial_{\bd{Q}_{\pm}} c \\
    & \hspace{6em} - \frac{\I c}{\abs{\bd{P}_{\pm}}} \Bigl(
    \frac{\bd{P}_{\pm}\otimes \bd{P}_{\pm}}{\abs{\bd{P}_{\pm}}^2} - I
    \Bigr) - \I \abs{\bd{P}_{\pm}}\partial_{\bd{Q}_{\pm}}^2 c \Bigr)
    \biggr).
\end{aligned}
\end{equation*}
We note that the solution associated with the first branch ($H_1 =
0$) does not involve in time, since the Hamiltonian flow
\eqref{eq:Hflow} is $\bd{Q}_1(t) \equiv \bd{q}, \ \bd{P}_1(t) \equiv
\bd{p}$ and $\sigma_1$ stays constant.

\section{High frequency wave propagation}\label{sec:HFWP}

The frozen Gaussian approximation (FGA) formulated in
Section~\ref{sec:FGAgen} approximates the propagation operator of
hyperbolic system. The approximation is useful especially in the case
of high frequency wave propagation, where the small parameter $\veps$
should be chosen according to the initial condition.

We rewrite the frozen Gaussian approximation \eqref{eq:FGAlg} as
\begin{equation}\label{eq:FGAreform}
  \begin{aligned}
    \bd{u}(t, \bd{x}) & = \frac{1}{(2\pi\veps)^{3d/2}} \sum_{m=1}^M
    \int \bd{a}_{0, m}(t, \bd{q}, \bd{p}) e^{\I \Phi_m / \veps}
    \bd{L}_m^{\TT}(\bd{q}, \bd{p}) \bd{u}_0(\bd{y}) \ud \bd{y} \ud
    \bd{p} \ud \bd{q} \\
    & = \frac{1}{(2\pi\veps)^{3d/2}} \sum_{m=1}^M \int \bd{a}_{0,
      m}(t, \bd{q}, \bd{p}) e^{\I (S_m + \I \abs{\bd{x} -
        \bd{Q}_m}^2/2
      + \bd{P}_m \cdot (\bd{x} - \bd{Q}_m)) / \veps} \\
    & \hspace{10em} \times \bd{L}_m^{\TT}(\bd{q}, \bd{p})
    \bd{\psi}_0(\bd{q}, \bd{p}) \ud \bd{p} \ud \bd{q},
  \end{aligned}
\end{equation}
with $\bd{\psi}_0(\bd{q}, \bd{p})$ given by
\begin{equation}\label{eq:initialproj}
  \bd{\psi}_0(\bd{q}, \bd{p}) = \int e^{- \abs{\bd{y} - \bd{q}}^2
    /(2\veps) - \I \bd{p}\cdot(\bd{y} - \bd{q})/\veps}
  \bd{u}_0(\bd{y}) \ud \bd{y}.
\end{equation}

This implies that the initial condition is first transformed to be
on the $(\bd{q}, \bd{p})$ phase plane by taking inner product with
Gaussian functions, then one evolves the centers $(\bd{Q}_m,
\bd{P}_m)$ and weights $\sigma_{m}$ of these Gaussian functions by
\eqref{eq:Hflow} and \eqref{eq:sigma}. At final time $t$, the
solution is approximated by superpositions of these Gaussian
functions. We address two issues appearing in numerical algorithms
of FGA: one is to estimate the mesh size of $(\bd{q}, \bd{p})$ in
the discretization of \eqref{eq:FGAreform}; the other is to discuss
the choice of small parameter $\veps$ when the initial condition
takes the form
\begin{equation}\label{ini:WKB}
  \bd{u}(0, \bd{x}) = \bd{A}_0(\bd{x}) e^{\I S_0(\bd{x}) / \eta},
\end{equation}
where $\bd{A}_0$ and $S_0$ are smooth and compact support functions.
The parameter $\eta$ characterizes the frequency of the initial
wave. Small $\eta$ indicates high frequency waves.

\subsection{Mesh size of $(\bd{q}, \bd{p})$}
Taking derivatives of \eqref{eq:initialproj} with respect to
$\bd{p}$ and $\bd{q}$ produces
\begin{align}
  \partial_{\bd{p}} \bd{\psi}_0(\bd{q}, \bd{p}) & = - \frac{\I}{\veps}
  \bd{\varphi}(\bd{q}, \bd{p}); \\
  \partial_{\bd{q}} \bd{\psi}_0(\bd{q}, \bd{p}) & = \frac{1}{\veps}
  \bd{\varphi}(\bd{q}, \bd{p}) + \frac{\I}{\veps} \bd{p} \otimes
  \bd{\psi}_0(\bd{q}, \bd{p}),
\end{align}
where
\begin{equation*}
  \bd{\varphi}(\bd{q}, \bd{p}) =
  \int (\bd{y} -  \bd{q}) \otimes e^{-\abs{\bd{y} - \bd{q}}^2/(2\veps)
    - \I \bd{p}\cdot(\bd{y} - \bd{q})/\veps} \bd{u}_0(\bd{y}) \ud \bd{y}.
\end{equation*}

Denote the integrand in \eqref{eq:FGAreform} as, where we drop the
subscript $m$ without loss of generality,
\begin{equation}\label{eq:Lambda}
  \bd{\Lambda}(t,\bd{x}, \bd{q}, \bd{p}) =
  \bd{a}_{0}(t, \bd{q}, \bd{p}) e^{\I (S + \I \abs{\bd{x} -
      \bd{Q}}^2/2
    + \bd{P} \cdot (\bd{x} - \bd{Q})) / \veps}  \bd{L}^{\TT}(\bd{q}, \bd{p})
  \bd{\psi}_0(\bd{q}, \bd{p}).
\end{equation}
Taking derivatives of \eqref{eq:Lambda} with respect to $\bd{p}$ and
$\bd{q}$ yields
\begin{equation*}
  \begin{aligned}
    \partial_{\bd{p}} \bd{\Lambda} & =
    \partial_{\bd{p}} \bd{a} \bd{L}^{\TT} \bd{\psi}_0 e^{\I (S + \I
      \abs{\bd{x} - \bd{Q}}^2/2 + \bd{P} \cdot (\bd{x} - \bd{Q})) /
      \veps} \\
    & + \biggl( \frac{\I}{\veps} \bigl( \partial_{\bd{p}} S +
    \I \partial_{\bd{p}} \bd{Q} \cdot (\bd{Q} - \bd{x})
    + \partial_{\bd{p}}\bd{P} \cdot (\bd{x} - \bd{Q})
    - \partial_{\bd{p}}\bd{Q} \cdot \bd{P}\bigr) \bd{L} \cdot \bd{\psi}_0  \\
    & \hspace{6em} +
    \partial_{\bd{p}} \bd{L} \cdot \bd{\psi}_0 - \frac{\I}{\veps}
    \bd{\varphi} \cdot \bd{L} \biggr) \otimes \bd{a} e^{\I (S + \I
      \abs{\bd{x} - \bd{Q}}^2/2 + \bd{P} \cdot (\bd{x} - \bd{Q})) /
      \veps},
  \end{aligned}
\end{equation*}
and
\begin{equation*}
  \begin{aligned}
    \partial_{\bd{q}} \bd{\Lambda} & =
    \partial_{\bd{q}} \bd{a} \bd{L}^{\TT} \bd{\psi}_0 e^{\I (S + \I
      \abs{\bd{x} - \bd{Q}}^2/2 + \bd{P} \cdot (\bd{x} - \bd{Q})) /
      \veps} \\
    & + \biggl( \frac{\I}{\veps} \bigl( \partial_{\bd{q}} S +
    \I \partial_{\bd{q}} \bd{Q} \cdot (\bd{Q} - \bd{x})
    + \partial_{\bd{q}}\bd{P} \cdot (\bd{x} - \bd{Q})
    - \partial_{\bd{q}}\bd{Q} \cdot \bd{P}\bigr) \bd{L} \cdot \bd{\psi}_0  \\
    & \hspace{3em} +
    \partial_{\bd{q}} \bd{L} \cdot \bd{\psi}_0 + \frac{1}{\veps}
    \bd{\varphi} \cdot \bd{L} + \frac{\I}{\veps} \bd{p} \bd{L} \cdot
    \bd{\psi}_0 \biggr) \otimes \bd{a} e^{\I (S + \I \abs{\bd{x} -
        \bd{Q}}^2/2 + \bd{P} \cdot (\bd{x} - \bd{Q})) / \veps}.
  \end{aligned}
\end{equation*}
By \eqref{eq:SqSp} in the proof of Lemma~\ref{lem:veps1}, we can
simplify the above expressions of derivatives as
\begin{equation*}
  \begin{aligned}
    \partial_{\bd{p}} \bd{\Lambda} & =
    \partial_{\bd{p}} \bd{a} \bd{L}^{\TT} \bd{\psi}_0 e^{\I (S + \I
      \abs{\bd{x} - \bd{Q}}^2/2 + \bd{P} \cdot (\bd{x} - \bd{Q})) /
      \veps} \\
    & + \biggl( \frac{\I}{\veps} ( \partial_{\bd{p}}\bd{P} -
    \I \partial_{\bd{p}} \bd{Q} ) \cdot (\bd{x} -
    \bd{Q})  \bd{L} \cdot \bd{\psi}_0  \\
    & \hspace{6em} +
    \partial_{\bd{p}} \bd{L} \cdot \bd{\psi}_0 - \frac{\I}{\veps}
    \bd{\varphi} \cdot \bd{L} \biggr) \otimes \bd{a} e^{\I (S + \I
      \abs{\bd{x} - \bd{Q}}^2/2 + \bd{P} \cdot (\bd{x} - \bd{Q})) /
      \veps},
  \end{aligned}
\end{equation*}
and
\begin{equation*}
  \begin{aligned}
    \partial_{\bd{q}} \bd{\Lambda} & =
    \partial_{\bd{q}} \bd{a} \bd{L}^{\TT} \bd{\psi}_0 e^{\I (S + \I
      \abs{\bd{x} - \bd{Q}}^2/2 + \bd{P} \cdot (\bd{x} - \bd{Q})) /
      \veps} \\
    & + \biggl( \frac{\I}{\veps}
    ( \partial_{\bd{q}}\bd{P} - \I \partial_{\bd{q}} \bd{Q})
    \cdot (\bd{x} - \bd{Q})
     \bd{L} \cdot \bd{\psi}_0  \\
    & \hspace{3em} +
    \partial_{\bd{q}} \bd{L} \cdot \bd{\psi}_0 + \frac{1}{\veps}
    \bd{\varphi} \cdot \bd{L} \biggr) \otimes \bd{a} e^{\I (S + \I
      \abs{\bd{x} - \bd{Q}}^2/2 + \bd{P} \cdot (\bd{x} - \bd{Q})) /
      \veps}.
  \end{aligned}
\end{equation*}
Keeping only the highest order terms gives
\begin{equation*}
\begin{aligned}
  \partial_{\bd{p}} \bd{\Lambda} & = \biggl( \frac{\I}{\veps}
  ( \partial_{\bd{p}}\bd{P} - \I \partial_{\bd{p}} \bd{Q} ) \cdot
  (\bd{x} - \bd{Q}) \bd{L} \cdot \bd{\psi}_0 - \frac{\I}{\veps}
  \bd{\varphi} \cdot \bd{L} \biggr) \otimes \bd{a} \\
  & \hspace{4em} \times e^{\I (S + \I \abs{\bd{x} - \bd{Q}}^2/2 +
    \bd{P} \cdot (\bd{x} -
    \bd{Q})) / \veps} + \Or(1),
\end{aligned}
\end{equation*}
and
\begin{equation*}
\begin{aligned}
  \partial_{\bd{q}} \bd{\Lambda} & = \biggl( \frac{\I}{\veps} \bigl(
  ( \partial_{\bd{q}}\bd{P} - \I \partial_{\bd{q}} \bd{Q}) \cdot
  (\bd{x} - \bd{Q}) \bd{L} \cdot \bd{\psi}_0 + \frac{1}{\veps}
  \bd{\varphi} \cdot \bd{L} \biggr) \otimes \bd{a} \\
  & \hspace{4em} \times e^{\I (S + \I \abs{\bd{x} - \bd{Q}}^2/2 +
    \bd{P} \cdot (\bd{x} - \bd{Q})) / \veps} + \Or(1).
\end{aligned}
\end{equation*}
Notice that $(\bd{x} - \bd{Q})$ and $(\bd{y} - \bd{q})$ is
$\Or(\veps^{1/2})$ due to the Gaussian factor, therefore both
derivatives are $\Or(\veps^{-1/2})$, while the function
$\bd{\Lambda}$ is $\Or(1)$. As a result, in order to get an accurate
discretization of the integral \eqref{eq:FGAreform}, one has to take
the mesh size in $\bd{q}$ and $\bd{p}$ to be at least
$\Or(\veps^{1/2})$.

\subsection{Choice of parameter $\veps$}\label{sec:para_choice}
While the original hyperbolic system lives in physical domain, FGA
works on phase plane, hence the dimensionality is doubled. The cost
of numerical algorithm based on FGA can be estimated by the number
of mesh points used on phase plane. This means we need to find the
region where $\bd{\psi}_0(\bd{q}, \bd{p})$ makes a significant
contribution. In this subsection we investigate the effect of
$\veps$ on the size of region under the consideration of the high
frequency initial condition \eqref{ini:WKB}.

Substitute the initial condition \eqref{ini:WKB} into
\eqref{eq:initialproj}, we have
\begin{equation*}
  \bd{\psi}_0(\bd{q}, \bd{p})  = \int e^{- \abs{\bd{y} - \bd{q}}^2 /(2\veps) - \I
    \bd{p}\cdot(\bd{y} - \bd{q})/\veps + \I S_0(\bd{y})/ \eta}
  \bd{A}_0(\bd{y}) \ud \bd{y}.
\end{equation*}
We will choose $\veps$ comparable to $\eta$ and discuss the effects
of increasing or decreasing $\veps$. Let $r = \veps/\eta$, then
\begin{equation*}
  \bd{\psi}_0(\bd{q}, \bd{p})  =
  \int e^{- \abs{\bd{y} - \bd{q}}^2/(2\veps)}
  e^{\I ( - \bd{p}\cdot (\bd{y}-\bd{q}) + r S_0(\bd{y}) )
    / \veps} \bd{A}_0(\bd{y}) \ud \bd{y}.
\end{equation*}
Taylor expansion gives
\begin{equation*}
  S_0(\bd{y}) = S_0(\bd{q}) + \nabla S_0(\bd{q}) \cdot (\bd{y} - \bd{q})
  + \frac{1}{2} \nabla^2 S_0(\bd{q} + \theta ( \bd{y} - \bd{q})) :
  (\bd{y} - \bd{q})^2,
\end{equation*}
where $\theta \in [0, 1]$ depends on $\bd{y}$.
Define $$R_0(\bd{y},
\bd{q}) = S_0(\bd{y}) - S_0(\bd{q}) - \nabla S_0(\bd{q}) \cdot (\bd{y}
- \bd{q}),$$ then we have
\begin{equation*}
  \begin{aligned}
    \bd{\psi}_0(\bd{q}, \bd{p}) & = e^{\I r S_0(\bd{q})/\veps} \int e^{-
      \abs{\bd{y} - \bd{q}}^2/(2\veps)} e^{\I (- \bd{p} + r\nabla
      S_0(\bd{q})) \cdot (\bd{y}-\bd{q})
      / \veps} e^{\I R_0(\bd{y}, \bd{q})/\veps} \bd{A}_0(\bd{y}) \ud \bd{y} \\
    & = \sqrt{\veps} e^{\I r S_0(\bd{q})/\veps} \int e^{- \abs{\bd{y}
      }^2/2} e^{\I (- \bd{p} + r\nabla S_0(\bd{q})) \cdot
      \bd{y} / \veps^{1/2}} \\
    & \hspace{10em} \times e^{\I R_0(\bd{q} + \sqrt{\veps}\bd{y},
      \bd{q})/\veps} \bd{A}_0(\bd{q} + \sqrt{\veps} \bd{y}) \ud
    \bd{y}.
  \end{aligned}
\end{equation*}
Define
\begin{equation*}
  \bd{f}_{\bd{q}}(\bd{y}) = e^{-\abs{\bd{y}}^2/2} e^{\I
    R_0(\bd{q} + \sqrt{\veps}\bd{y}, \bd{q})/\veps} \bd{A}_0(\bd{q} +
  \sqrt{\veps} \bd{y}).
\end{equation*}
By the definition of $R_0$, it is clear that the derivative of
$\bd{f}_{\bd{q}}$ with respect to $\bd{y}$ is bounded independent of
$\veps$. Therefore, by
\begin{equation}\label{eq:psi0}
  \bd{\psi}_0(\bd{q}, \bd{p}) = \sqrt{\veps} e^{ir S_0(\bd{q})/\veps}
  \wh{\bd{f}}_{\bd{q}}(\veps^{-1/2}(\bd{p} - r\nabla S_0(\bd{q}))),
\end{equation}
standard integration by parts argument yields
\begin{equation}\label{eq:psi0Est}
  \abs{\bd{\psi}_0(\bd{q}, \bd{p})} \leq
  \frac{C}{\abs{\bd{p} - r \nabla S_0(\bd{q})}^{N}} \veps^{(N+1)/2},
\end{equation}
for any positive integer $N$. In \eqref{eq:psi0},
$\wh{\bd{f}}_{\bd{q}}$ means the Fourier transform of $\bd{f_q}$, and
\eqref{eq:psi0Est} is actually the decay rate of the Fourier
transform.

The equation \eqref{eq:psi0Est} implies $\bd{\psi}_0(\bd{q}, \bd{p})$ makes
significant contributions only when $\abs{\bd{p} - r \nabla
  S_0(\bd{q})}$ is $\Or(\veps^{1/2})$. Therefore one only needs to
consider the region where $\bd{p}$ is localized around $r\nabla
S_0(\bd{q})$. If one takes the mesh size of $\bd{p}$ as
$\Or(\veps^{1/2})$, then the number of mesh points in $\bd{p}$ given
$\bd{q}$ is a constant. The total number of $(\bd{q}, \bd{p})$ points
to be considered is $\Or(\veps^{-d/2})\sim \Or(\eta^{-d/2})$.
Therefore while smaller $\veps$ gives better asymptotic accuracy, it
requires more computation cost. This is a trade-off between cost and
performance.

It is seen that FGA is suitable for computing high frequency wave
propagation when $\veps$ is taken comparable to $\eta$, which is the
reciprocal of the characteristic frequency of initial wave field. We
remark that, in a follow-up work \cite{LuYang:CPAM}, we establish
rigorous analysis on the accuracy of FGA for high frequency wave
propagation for general linear strictly hyperbolic system.

\section{Eulerian Frozen Gaussian approximation}\label{sec:EFGA}
In this section we introduce Eulerian frozen Gaussian approximation
(EFGA) for computation of linear strictly hyperbolic system. We first
describe Eulerian formulation, followed by numerical algorithms based
on the Eulerian formulation. These Eulerian methods can also be
applied for computation of the Herman-Kluk propagator \cite{HeKl:84} in
quantum mechanics, which is discussed in the last subsection.

\subsection{Eulerian formulation}\label{sec:EuFormulation} The formulation of
EFGA is given by
\begin{equation}\label{Euler:sum}
  \bd{u}^{\EFGA}(t, \bd{x}) = \frac{1}{(2\pi \veps)^{3d/2}} \sum_{m=1}^M
  \int {\sigma}_{m}(t, \bd{Q}, \bd{P})\bd{R}_m(\bd{Q}, \bd{P})e^{\I\Theta_m/\veps}
\ud \bd{P} \ud \bd{Q},
\end{equation}
where the phase function $\Theta_m$ is
\begin{equation}\label{Euler:Thetam}
  \Theta_m(t, \bd{x}, \bd{Q}, \bd{P})
  = S_m(t, \bd{Q}, \bd{P}) + \bd{P}\cdot(\bd{x} - \bd{Q})
  + \frac{\I}{2} \abs{\bd{x} - \bd{Q}}^2.
\end{equation}
Define the Liouville operator
\begin{equation*}
\LL_m=\partial_t+\partial_{\bd{P}}H_m\cdot\partial_{\bd{Q}}
-\partial_{\bd{Q}}H_m\cdot\partial_{\bd{P}}.
\end{equation*}
The evolution of $S_m(t,\bd{Q}, \bd{P})$ satisfies
\begin{equation}\label{Euler:S}
\LL_mS_m=\bd{P}\cdot\partial_{\bd{P}}H_m-H_m,
\end{equation}
with initial condition $S_m(0,\bd{Q}, \bd{P})=0$.

To get the evolution of $\sigma_m$, we define the
auxiliary functions 
$$\bd{\phi}_m(t,\bd{Q}, \bd{P})
=(\phi_{m,1},\cdots,\phi_{m,d}),$$ 
given by
\begin{equation}\label{Euler:phi}
\LL_m\bd{\phi}_m=0,
\end{equation}
with initial condition
\begin{equation}
\bd{\phi}_m(0,\bd{Q}, \bd{P})=\bd{P}+\I \bd{Q}.
\end{equation}
Once $\bd{\phi}_m$ is determined, the evolution of $\sigma_m(t,\bd{Q}, \bd{P})$
is given by (where
we have omitted the subscript $m$ for simplicity of notation)
\begin{equation}\label{Euler:sigma}
\begin{aligned}
\LL\sigma=&-\sigma \bd{L}^{\TT} (\partial_{\bd{P}} H
  \cdot \partial_{\bd{Q}} \bd{R} - \partial_{\bd{Q}} H
  \cdot \partial_{\bd{P}} \bd{R}) \\ & - \sigma
  (\partial_{\bd{P}}\phi_k\cdot \partial_{\bd{Q}}\bd{L}-
  \partial_{\bd{Q}}\phi_k\cdot \partial_{\bd{P}}\bd{L})^{\TT} \bd{F}_{j} Z_{jk}^{-1}
  \\& \qquad\qquad - \sigma \partial_{P_j} \phi_{n} Z_{kn}^{-1} \bd{L}^{\TT} \bigl(
  - \partial_{Q_{j}} A_k + \frac{\I}{2}
  P_{l} \partial_{Q_{j}} \partial_{Q_{k}} A_l \bigr) \bd{R},
\end{aligned}
\end{equation}
with initial condition
\begin{equation}\label{ini:sigma}
\sigma_m(0,\bd{Q}, \bd{P})=2^{d/2}\int v_{m,0}(\bd{y}, \bd{Q},
\bd{P})\exp\bigg(\frac{\I}{\veps}\big(-\bd{P}
\cdot(\bd{y}-\bd{Q})+\frac{\I}{2}\abs{\bd{y}-\bd{Q}}^2\big)\bigg)
\ud\bd{y},
\end{equation}
where $v_{m,0}(\bd{y}, \bd{Q}, \bd{P})=\bd{L}_m^{\TT}(\bd{Q},
\bd{P}) \bd{u}_0(\bd{y})$. In \eqref{Euler:sigma}, we have used the
shorthand notations,
\begin{align}
    & Z=\bigl( \partial_{\bd{P}}\bd{\phi} \bigr)^{\TT}-\I \bigl( \partial_{\bd{Q}}\bd{\phi}
    \bigr)^{\TT}, \\
    &\bd{F}_{j} = - \Bigl( (A_j - \partial_{P_{j}} H) + \I ( \partial_{Q_{j}} H
    - P_{l} \partial_{Q_{j}} A_l ) \Bigr) \bd{R}.
\end{align}

\subsection{Derivation}


Under the change of variable,
\begin{equation}
\eta_m: \quad\begin{array}{ccc} \RR^d\times\RR^d & \longrightarrow &
\RR^d\times\RR^d \\ (\bd{q}, \bd{p}) & \longrightarrow & (\bd{Q},
\bd{P})
\end{array},
\end{equation}
where $\eta_m$ is the Hamiltonian flow given by
\eqref{eq:Hflow}-\eqref{ini:Hflow}, the FGA formulation
\eqref{eq:FGAlg} can be rewritten as
\begin{equation*}
  \bd{u}^{\EFGA}(t, \bd{x}) = \frac{1}{(2\pi \veps)^{3d/2}} \sum_{m=1}^M
  \int \bd{a}_{m}(t, \bd{Q}, \bd{P})e^{\I\Phi_m/\veps}
  v_{m,0}(\bd{y}, \bd{q}_m, \bd{p}_m)
  \ud \bd{y} \ud \bd{P} \ud \bd{Q},
\end{equation*}
where $(\bd{q}_m, \bd{p}_m)=\eta^{-1}_m\bigl(\bd{Q}, \bd{P}\bigr)$ and
we have used the fact that the Jacobian of $\eta_m$ equals to $1$ due
to symplecticity of the Hamiltonian flow. 

The phase function $\Phi_m$ is given by
\begin{multline}\label{Euler:Phim}
  \Phi_m(t, \bd{x}, \bd{y}, \bd{Q}, \bd{P})
  = S_m(t, \bd{Q}, \bd{P}) + \frac{\I}{2} \abs{\bd{x} - \bd{Q}}^2
  + \bd{P}\cdot(\bd{x} - \bd{Q}) \\ + \frac{\I}{2} \abs{\bd{y} - \bd{q}_m}^2
  - \bd{p}_m\cdot(\bd{y} - \bd{q}_m).
\end{multline}

\begin{remark}
  In the Lagrangian formulation of FGA, solution of each branch starts
  at the same $(\bd{q}, \bd{p})$, so that $(\bd{q}, \bd{p})$ is
  independent of $m$ while $(\bd{Q}, \bd{P})$ given by \eqref{eq:Hflow}
  depends on $m$; in the Eulerian formulation of FGA, solution of each
  branch ends at the same $(\bd{Q}, \bd{P})$, therefore
  $(\bd{Q}, \bd{P})$ is independent of $m$ while $(\bd{q}, \bd{p})$
  given by $\eta^{-1}_m$ depends on $m$.
\end{remark}

What remains is to derive evolution equations of
$\bd{a}_m$ and $S_m$ in terms of Eulerian coordinates $(\bd{Q},
\bd{P})$. The easy observation is to change time derivative in
Lagrangian coordinate to the Liouville operator in Eulerian coordinate
by chain rule and \eqref{eq:Hflow},
\[\frac{\ud}{\ud t}\longrightarrow \LL_m=\partial_t
+\partial_{\bd{P}}H_m\cdot\partial_{\bd{Q}}
-\partial_{\bd{Q}}H_m\cdot\partial_{\bd{P}},\] which, together with
\eqref{eq:S}, implies \eqref{Euler:S}.

The difficult part is, in the evolution equation \eqref{eq:sigma}
for $\bd{a}_m$, there
are terms containing Lagrangian derivatives with respect to $\bd{q}$
and $\bd{p}$. To replace them in Eulerian coordinate, one needs the
following theorem, where we omit the subscript $m$ for simplicity.

\begin{theorem}\label{thm:deriv}
Assume $\bd{\phi}(t,\bd{Q}, \bd{P})=(\phi_1,\cdots,\phi_d)$ is the
solution to
\begin{equation}\label{eq:levelphi}
\partial_t\bd{\phi}+\partial_{\bd{P}}H\cdot\partial_{\bd{Q}}\bd{\phi}
-\partial_{\bd{Q}}H\cdot\partial_{\bd{P}}\bd{\phi}=0,
\end{equation}
with initial condition
\begin{equation}\label{eq:E3}
  \bd{\phi}(0,\bd{Q}, \bd{P})=\bd{P}+\I\bd{Q}.\end{equation}

Denote
\[X=\partial_{\bd{z}}\bd{Q},\qquad Y=\partial_{\bd{z}}\bd{P},\]
where $\partial_{\bd{z}}\bd{Q}$ and $\partial_{\bd{z}}\bd{P}$ are
given in Lagrangian coordinate, then
\begin{equation}\label{eq:sub}X=\big(\partial_{\bd{P}}\bd{\phi}\big)^{\TT}, \quad
Y=-\big(\partial_{\bd{Q}}\bd{\phi}\big)^{\TT},\end{equation} in
Eulerian coordinate.
\end{theorem}
\begin{proof}
Differentiating \eqref{eq:Hflow} with respect to $\bd{z}$ produces
\begin{equation}\label{eq:ODEder}
\begin{cases}
& \dsp \frac{\ud X}{\ud t}=X\frac{\partial^2
H}{\partial\bd{Q}\partial\bd{P}}+Y\frac{\partial^2H}{\partial\bd{P}^2},
\\[1em]  &  \dsp \frac{\ud Y}{\ud t}=-X\frac{\partial^2
H}{\partial\bd{Q}^2}-Y\frac{\partial^2H}{\partial\bd{P}\partial\bd{Q}},
\end{cases}
\end{equation}
which are the equations for $X$ and $Y$ in Lagrangian coordinate.

Therefore $X$ and $Y$ satisfy, in Eulerian coordinate,
\begin{equation}\label{eq:E1}
\begin{cases}
& \dsp \LL X=X\frac{\partial^2
H}{\partial\bd{Q}\partial\bd{P}}+Y\frac{\partial^2H}{\partial\bd{P}^2},
\\[1em]  &  \dsp \LL Y=-X\frac{\partial^2
H}{\partial\bd{Q}^2}-Y\frac{\partial^2H}{\partial\bd{P}\partial\bd{Q}}.
\end{cases}
\end{equation}

Differentiating \eqref{eq:levelphi} with respect to $\bd{P}$ and
$\bd{Q}$ yields
\begin{equation}\label{eq:E2}
\begin{cases}
& \dsp \LL \big( \partial_{\bd{P}}\bd{\phi} \big)=-\frac{\partial^2
H}{\partial \bd{P}^2}\partial_{\bd{Q}}\bd{\phi}+\frac{\partial^2
H}{\partial\bd{P}\partial\bd{Q}}\partial_{\bd{P}}\bd{\phi}, \\[1em]
& \dsp \LL \big( \partial_{\bd{Q}}\bd{\phi} \big)=-\frac{\partial^2
H}{\partial
\bd{Q}\partial\bd{P}}\partial_{\bd{Q}}\bd{\phi}+\frac{\partial^2
H}{\partial\bd{Q}^2}\partial_{\bd{P}}\bd{\phi}.
\end{cases}
\end{equation}

Comparing \eqref{eq:E1} with \eqref{eq:E2}, one can see that $X$ and
$Y$ satisfy the same the Liouville equations as
$\big(\partial_{\bd{P}}\bd{\phi}\big)^{\TT}$ and
$-\big(\partial_{\bd{Q}}\bd{\phi}\big)^{\TT}$. Moreover,
\eqref{eq:E3} implies that $X$ and $Y$ have the same initial
conditions as $\big(\partial_{\bd{P}}\bd{\phi}\big)^{\TT}$ and
$-\big(\partial_{\bd{Q}}\bd{\phi}\big)^{\TT}$. Since the Liouville
equation is linear, we have in Eulerian coordinate,
\begin{equation*}X=\big(\partial_{\bd{P}}\bd{\phi}\big)^{\TT}, \quad
Y=-\big(\partial_{\bd{Q}}\bd{\phi}\big)^{\TT}.\end{equation*}
\end{proof}

\begin{remark}
We observe that \eqref{eq:ODEder} is equivalent to the dynamic ray
tracing equations in Gaussian beam method \cite{Po:82}. Therefore
Theorem~\ref{thm:deriv} can be applied in computing Hessian
functions in Eulerian Gaussian beam methods, which is essentially
the approach used in
\cites{JiWuYaHu:10,JiWuYa:08,JiWuYa:10,JiWuYa:11}.
\end{remark}

Hence \eqref{eq:sigma} and \eqref{eq:sub} imply
\eqref{Euler:sigma}.

Rewrite $\Phi_m$ in \eqref{Euler:Phim} as
\[\Phi_m=\Theta_m+ \frac{\I}{2} \abs{\bd{y} - \bd{q}_m}^2
  - \bd{p}_m\cdot(\bd{y} - \bd{q}_m),\]
where $\Theta_m$ is given by \eqref{Euler:Thetam}. Since $\bd{q}$
and $\bd{p}$ are parameters in \eqref{eq:sigma}, the initial
condition \eqref{ini:sigma} is obtained by combining the term
\[ \int v_{m,0}(\bd{y}, \bd{q}, \bd{p})e^{-
  \frac{\I}{\veps}\bd{p}_m\cdot(\bd{y} - \bd{q}_m) - \frac{1}{2\veps}
  \abs{\bd{y} - \bd{q}_m}^2} \ud\bd{y}\] with the initial condition
\eqref{ini:sigma_la}.

\subsection{Algorithm}\label{sec:algorithm}

We introduce two numerical algorithms based on the Eulerian
formulation: Eulerian method and semi-Lagrangian method. Let us first
describe the meshes needed in these algorithms.

\subsubsection{Numerical meshes}\label{sec:mesh}

\begin{enumerate}
\item Discrete meshes of $\bd{Q}$ and $\bd{P}$ for solving the
  Liouville equations.

Denote $\bd{\delta Q}=(\delta Q_1,\cdots,\delta Q_d)$ and
$\bd{\delta P}=(\delta P_1,\cdots,\delta P_d)$ as the mesh size.
Suppose $\bd{Q}^0=(Q^0_1,\cdots,Q^0_d)$ is the starting point, then
the mesh grids $\bd{Q^k}$, $\bd{k}=(k_1,\cdots,k_d)$, are defined as
\[\bd{Q^k}=\bigl(Q^0_1+(k_1-1)\delta Q_1,\cdots, Q^0_d+(k_d-1)\delta Q_d\bigr),\]
where $k_j=1,\cdots,N_q$ for each $j\in\{1,\cdots,d\}$.

The mesh grids $\bd{P^\ell}$, $\bd{\ell}=(\ell_1,\cdots,\ell_d)$,
are defined as
\[\bd{P^\ell}=\bigl(P^0_1+(\ell_1-1)\delta P_1,\cdots, P^0_d+(\ell_d-1)\delta P_d\bigr),\]
where $\ell_j=1,\cdots,N_p$ for each $j\in\{1,\cdots,d\}$ and
$\bd{P}^0=(P^0_1,\cdots,P^0_d)$ is the starting point.

\item Discrete mesh of $\bd{y}$ for evaluating the initial condition
in \eqref{ini:sigma}. $\bd{\delta y}=(\delta y_1,\cdots,\delta y_d)$
is the mesh size. Denote $\bd{y}^0=(y^0_1,\cdots,y^0_d)$ as the
starting point. The mesh grids $\bd{y^m}$ are,
$\bd{m}=(m_1,\cdots,m_d)$,
\[\bd{y^m}=\bigl(y^0_1+(m_1-1)\delta y_1,\cdots, y^0_d+(m_d-1)\delta y_d\bigr),\]
where $m_j=1,\cdots,N_y$ for each $j\in\{1,\cdots,d\}$.

\item Discrete mesh of $\bd{x}$ for reconstructing the final solution.
$\bd{\delta x}=(\delta x_1,\cdots,\delta x_d)$ is the mesh size.
Denote $\bd{x}^0=(x^0_1,\cdots,x^0_d)$ as the starting point. The
mesh grids $\bd{x^n}$ are, $\bd{n}=(n_1,\cdots,n_d)$,
\[\bd{x^n}=\bigl(x^0_1+(n_1-1)\delta x_1,\cdots, x^0_d+(n_d-1)\delta x_d\bigr),\]
where $n_j=1,\cdots,N_x$ for each $j\in\{1,\cdots,d\}$.

\end{enumerate}

\subsubsection{Eulerian method}
With the prepared meshes, the Eulerian frozen Gaussian approximation
algorithm is given as follows.

\begin{itemize}

\item[Step 1.] Compute $\bd{L}_m(\bd{Q^k}, \bd{P^\ell})$ and
  $\bd{R}_m(\bd{Q^k}, \bd{P^\ell})$ by solving the eigenvalue problems
  \eqref{eigen:left}-\eqref{eigen:right}.

\item[Step 2.] Compute the initial condition \eqref{ini:sigma} at
$(\bd{Q^k}, \bd{P^\ell})$,
  \begin{multline}
    \sigma_{m}(0,\bd{Q^k}, \bd{P^\ell})=2^{d/2}\sum_{\bd{m}}
    e^{\frac{\I}{\veps}(- \bd{P^{\ell}} \cdot ( \bd{y^m} - \bd{Q^k}
      ) +\frac{\I}{2} \abs{\bd{y^{m}} - \bd{Q^k}}^2)}\\\times
    v_{m,0}(\bd{y^m}, \bd{Q^k}, \bd{P^\ell}) r_{\theta}(\abs{\bd{y^{m}} -
      \bd{Q^k}})\delta y_1\cdots \delta y_d,
  \end{multline}
  where $r_{\theta}$ is a cutoff function such that $r_{\theta} = 1$
  in the ball of radius $\theta>0$ centered at origin and $r_{\theta}
  = 0$ outside the ball.

\item[Step 3.] Solve \eqref{Euler:S}, \eqref{Euler:phi} and
  \eqref{Euler:sigma} by finite difference/volume/element methods, for
  example, standard upwind scheme with van Leer flux limiter
  (\cite{Le:92}). Denote the numerical solutions as
  $S_{m}^{\bd{k,\ell}}$ and $\sigma_{m}^{\bd{k,\ell}}$.

\item[Step 4.] Reconstruct the solution by \eqref{Euler:sum},
\begin{equation}
  \begin{aligned}
    \bd{u}^{\EFGA}(t, \bd{x^{n}}) & = \sum_{m=1}^M\sum_{\bd{k,\ell}} \biggl(
    \frac{\sigma_m^{\bd{k,\ell}}}{(2\pi \veps)^{3d/2}} e^{\frac{\I}{\veps}
      \big(S_m^{\bd{k,\ell}}+\bd{P^\ell}\cdot(\bd{x}^{\bd{n}} -
      \bd{Q^k})\big) - \frac{1}{2\veps} \abs{\bd{x}^{\bd{n}}
        - \bd{Q^k}}^2}    \biggr) \\
    & \qquad \times \bd{R}_m(\bd{Q^k}, \bd{P^\ell}) r_{\theta}(\abs{\bd{x}^{\bd{n}} -
      \bd{Q^k}})\delta Q_1\cdots \delta Q_d \delta
    P_1\cdots \delta P_d.
  \end{aligned}
\end{equation}

\end{itemize}

Note that a naive implementation of the above method will result in
numerical methods on the phase plane, and hence doubles the
dimensionality. A more efficient way is to implement this method
locally on the phase plane, when initial conditions have localization
properties. This local solver strategy is important to
make Eulerian methods efficient, and we detail the algorithm below.

If one considers WKB initial condition for linear strictly hyperbolic
system \eqref{eq:hypersys},
\begin{equation}
  \bd{u}(0,\bd{x})=\bd{a}_0(\bd{x})e^{\frac{\I}{\veps}S_0(\bd{x})},
\end{equation}
then the initial condition \eqref{ini:sigma} is localized in
momentum space around the submanifold
$\bd{P}=\nabla_{\bd{Q}}S_0(\bd{Q})$ as discussed in
Section~\ref{sec:para_choice}. A one-dimensional example is given in
Figure~\ref{fig:meshpq}. This localization property allows efficient
local implementation of the Eulerian method. One possible and simple
strategy is based on indicator functions. The idea is similar to the
moving mesh algorithm \cite{HaGoHa:91}.

\begin{figure}[htp]
  \begin{center}
    \resizebox{3.6in}{!}{\includegraphics{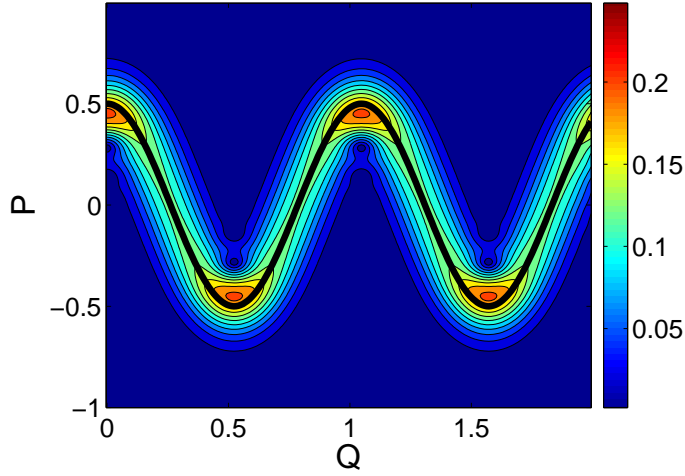}}
  \end{center}
  \caption{An illustration of the localization of $\sigma(0,Q,P)$ in
    one dimension for
    $v_{0}(y)=\frac{\sqrt{2}}{2}\exp\left(\I\frac{\sin(6y)}{12\veps}\right)$,
    $\veps=1/128$; the black solid curve is
    $P=\cos(6Q)/2$.}\label{fig:meshpq}
\end{figure}

Define indicator functions $\kappa_m$, $m=1,\cdots, M$, which satisfy
\begin{equation}
  \LL_m\kappa_m=0,
\end{equation}
with initial condition
\begin{equation}
  \kappa_m(0,\bd{Q}, \bd{P})=\left\{\begin{array}{ll} 1 &
      \quad\hbox{if}\;
      \sigma_m(0,\bd{Q}, \bd{P})\neq 0, \\
      0 & \quad\hbox{otherwise}.\end{array}\right.
\end{equation}
Then in Step 3 of the algorithm, when solving \eqref{Euler:S},
\eqref{Euler:phi} and \eqref{Euler:sigma} one only needs to update
function values on those $(\bd{Q^k}, \bd{P^\ell})$ where $\kappa_m$
is nonzero.

\begin{remark}

  1. In setting up the meshes, we assume that initial
  condition either has compact support or decays
  sufficiently fast to zero as $\bd{x}\rightarrow \infty$ so that we
  only need finite number of mesh points in space.

  2. The role of the truncation function $r_{\theta}$ is to save
  computational cost, since although Gaussian function is not
  localized, it decays quickly away from the center. In practice we
  take $\theta=\Or(\sqrt{\veps})$, the same order as the width of each
  Gaussian, when evaluate \eqref{ini:sigma} and \eqref{Euler:sum}
  numerically.

  3. There are two types of errors present in the method. The first
  type comes from the asymptotic approximation to strictly hyperbolic
  system. This error can {\it not} be reduced unless one includes
  higher order asymptotic corrections. The other type is the numerical
  error which comes from two sources: one is from solving the Liouville
  equations numerically; the other is from the discrete approximation
  of integrals \eqref{ini:sigma} and \eqref{Euler:sum}. It {\it can}
  be reduced by either taking small mesh size and time step or using
  higher order numerical methods.

  4. Step $2$ and $4$ can be expedited by making use of discrete
  fast Gaussian transform, as in \cites{QiYi:app1, QiYi:app2}.

\end{remark}

\subsubsection{Semi-Lagrangian method}

Alternatively, one can also use semi-Lagrangian method. This is a
type of Lagrangian method based on the Liouville equations
\eqref{Euler:S}, \eqref{Euler:phi} and \eqref{Euler:sigma}, which
can be viewed as a local implementation of Eulerian method on an
adaptive mesh. Different from the meshes of Eulerian method in
Section \ref{sec:mesh}, the meshes of $(\bd{Q}, \bd{P})$ in
semi-Lagrangian method is determined adaptively from initial
conditions, while the meshes of $\bd{y}$ and $\bd{x}$ are still the
same.

The underlying idea is to first lay down uniform mesh grids of
$(\bd{Q}, \bd{P})$, evolve the grids to time $t$ according to
Hamiltonian flow (denoted by $(\bd{Q}^t,\bd{P}^t)$); then set up
uniform mesh grids of $(\bd{Q}, \bd{P})$ at time $t$ based on
$(\bd{Q}^t,\bd{P}^t)$ and use method of characteristics to compute
the solutions to the Liouville equations. The difference from Eulerian
method is that it solves the Liouville equations by numerical
integrators for ODE instead of numerical schemes for PDE. The
detailed algorithm is given as follows, where we only focus on
computation of one eigenvalue branch and omit the subscript $m$ for
simplicity and clarity.

\begin{itemize}

\item[Step 1.] Choose initial uniform mesh grids
  $(\bd{Q}^{\bd{s},0}, \bd{P}^{\bd{r},0})$ where
  $\sigma(0,\bd{Q}^{\bd{s},0}, \bd{P}^{\bd{r},0})$ is nonzero.  Solve
  time-forward Hamiltonian flow
  \begin{equation*}
    \begin{cases}
    \dsp\frac{\ud \bd{Q}}{\ud t} = \partial_{\bd{P}} H(\bd{Q}, \bd{P}),
    \\[1em]
    \dsp\frac{\ud \bd{P}}{\ud t} = - \partial_{\bd{Q}} H(\bd{Q}, \bd{P}),
  \end{cases}
\end{equation*}
with initial conditions
\begin{equation*}
  \bd{Q}(0, \bd{Q}^{\bd{s},0}, \bd{P}^{\bd{r},0}) = \bd{Q}^{\bd{s},0}, \quad
  \text{and} \quad \bd{P}(0, \bd{Q}^{\bd{s},0}, \bd{P}^{\bd{r},0}) = \bd{P}^{\bd{r},0}.
\end{equation*}
Denote $\big(\bd{Q}^{\bd{s},t}, \bd{P}^{\bd{r},t}\big)= \big(
\bd{Q}(t,\bd{Q}^{\bd{s},0}, \bd{P}^{\bd{r},0}),
\bd{P}(t,\bd{Q}^{\bd{s},0}, \bd{P}^{\bd{r},0})\big)$.

\item[Step 2.] Choose uniform mesh grids
  $(\bd{Q^k}, \bd{P^\ell})$ so that all the points
  $\big(\bd{Q}^{\bd{s},t}, \bd{P}^{\bd{r},t}\big)$ lie in mesh cells.
  Solve time-backward Hamiltonian flow
  \begin{equation*}
    \begin{cases}
      \dsp\frac{\ud \bd{Q}}{\ud t} = - \partial_{\bd{P}} H(\bd{P},
      \bd{Q}),
      \\[1em]
      \dsp\frac{\ud \bd{P}}{\ud t} = \partial_{\bd{Q}} H(\bd{P},
      \bd{Q}),
    \end{cases}
  \end{equation*}
  with initial conditions
  \begin{equation*}
    \bd{Q}(0, \bd{Q^k}, \bd{P^\ell}) = \bd{Q^k}, \quad
    \text{and} \quad \bd{P}(0, \bd{Q^k}, \bd{P^\ell}) = \bd{P^\ell}.
  \end{equation*}
  In the meantime, solve time-backward equation
  \begin{equation*}
    \frac{\ud \wt{S}}{\ud t}=-\bd{P}\cdot\partial_{\bd{P}}H+H,
  \end{equation*}
  with initial condition
  $\wt{S}(0,\bd{Q^k}, \bd{P^\ell})=0$, then
  \[
  S(t,\bd{Q^k}, \bd{P^\ell})=-\wt{S}(t,\bd{Q^k}, \bd{P^\ell}).
  \]
  Denote $\big(\bd{q^k}, \bd{p^\ell}\big)=
  \big(\bd{Q}(t,\bd{Q^k}, \bd{P^\ell}),
  \bd{P}(t,\bd{Q^k}, \bd{P^\ell})\big)$, then
  \[
  \bd{\phi}(t,\bd{Q^k}, \bd{P^\ell})=\bd{p^\ell}+\I\bd{q^k}.
  \]

\item[Step 3.] Solve time-forward equation
  \begin{multline*}
    \frac{\ud \wt{\sigma}}{\ud t} + \wt{\sigma} \bd{L}^{\TT}
    (\partial_{\bd{P}} H \cdot \partial_{\bd{Q}} \bd{R}
    - \partial_{\bd{Q}} H \cdot \partial_{\bd{P}} \bd{R}) +
    \wt{\sigma}
    (\partial_{z_k} \bd{L})^{\TT} \bd{F}_{j} Z_{jk}^{-1}  \\
    + \wt{\sigma} \partial_{z_n} Q_{j} Z_{kn}^{-1} \bd{L}^{\TT} \bigl(
    - \partial_{Q_{j}} A_k + \frac{\I}{2}
    P_{l} \partial_{Q_{j}} \partial_{Q_{k}} A_l \bigr) \bd{R} = 0,
  \end{multline*}
  where
  \begin{equation*}
    \bd{F}_{j} = - \Bigl( (A_j - \partial_{P_{j}} H) + \I ( \partial_{Q_{j}} H
    - P_{l} \partial_{Q_{j}} A_l ) \Bigr) \bd{R},
  \end{equation*}
  with initial condition
  \begin{multline*}
    \wt{\sigma}(0,\bd{q^k}, \bd{p^\ell})=2^{d/2}\int
    v_{0}(\bd{y}, \bd{q^k}, \bd{p^\ell})\\
    \times\exp\bigg(\frac{\I}{\veps}\big(-\bd{p^\ell} \cdot
    (\bd{y}-\bd{q^k})
    +\frac{\I}{2}\abs{\bd{y}-\bd{q^k}}^2\big)\bigg)
    \ud\bd{y}.
  \end{multline*}
  Then
  $\sigma(t,\bd{Q^k}, \bd{P^\ell})
  =\wt{\sigma}(t,\bd{q^k}, \bd{p^\ell})$.

\item[Step 4.] Reconstruct the solution by
  \begin{equation*}
    \begin{aligned}
      \bd{u}^{\SLFGA}(t, \bd{x^{n}}) & = \sum_{\bd{k,\ell}} \biggl(
      \frac{\sigma^{\bd{k,\ell}}}{(2\pi \veps)^{3d/2}}
      e^{\frac{\I}{\veps}
        \big(S^{\bd{k,\ell}}+\bd{P}^{\bd{\ell}}\cdot(\bd{x}^{\bd{n}}
        - \bd{Q^k})\big) - \frac{1}{2\veps}
        \abs{\bd{x}^{\bd{n}}
          - \bd{Q^k}}^2}    \biggr) \\
      & \qquad \times \bd{R}(\bd{Q^k}, \bd{P^\ell})
      r_{\theta}(\abs{\bd{x}^{\bd{n}} -
        \bd{Q^k}})\delta Q_1\cdots
      \delta Q_d \delta P_1\cdots \delta P_d,
    \end{aligned}
  \end{equation*}
  where $\bd{R}$ the right eigenfunction in \eqref{eigen:right}.
\end{itemize}

\begin{remark}
  In Step 3 one needs to compute
  $Z=\partial_{\bd{z}}(\bd{Q}+\I\bd{P})$. Since
  $(\bd{q^k}, \bd{p^\ell})$ is not a uniform mesh grid,
  it can lose accuracy by using divided difference to compute
  $\partial_{\bd{z}}\bd{Q}$ and $\partial_{\bd{z}}\bd{P}$. To resolve
  this problem, one can solve \eqref{eq:ODEder} to get
  $\partial_{\bd{z}}\bd{Q}$ and $\partial_{\bd{z}}\bd{P}$ instead.
\end{remark}

\subsection{Eulerian method for the Herman-Kluk propagator of the Schr\"odinger
equation}\label{sec:SchrExa}

The Eulerian method introduced in Section~\ref{sec:EuFormulation} can
be also applied to computation of the Herman-Kluk propagator in quantum
mechanics.

The rescaled linear Schr\"odinger equation is given by
\begin{equation}\label{eq:Schr}
  i\veps\frac{\partial \Psi^\veps}{\partial
    t}=-\frac{\veps^2}{2}\Delta\Psi^\veps+U(\bd{x})\Psi^\veps, \quad
  \bd{x}\in\mathbb{R}^d \,,
\end{equation}
where $\Psi^\veps(t,\bd{x})$ is the wave function, $U(\bd{x})$ is the
potential and $\veps$ is the re-scaled Plank constant that describes
the ratio between quantum time/space scale and the macroscopic
time/space scale. This scaling corresponds to the semiclassical
regime.

We briefly describe the formulation of the Herman-Kluk propagator
\cite{HeKl:84} below. One can see that it is similar to the frozen
Gaussian approximation. In fact, the frozen Gaussian approximation
introduced in \cite{LuYang:CMS} is motivated by the ideas of the
Herman-Kluk propagator.  We remark that semiclassical approximation
underlying the Herman-Kluk propagator has been recently rigorously
analyzed by Swart and Rousse \cite{SwRo:09} and Robert \cite{Ro:09}.

Using the Herman-Kluk propagator, the solution to the Schr\"odinger
equation is approximated by
\begin{equation}
  \Psi^{\veps}_{\mathrm{HK}}(t, \bd{x}) = \frac{1}{(2\pi \veps)^{3d/2}}
  \int a(t, \bd{q}, \bd{p})e^{\I\Phi(t, \bd{x}, \bd{y}, \bd{q}, \bd{p})
  /\veps} \Psi_{0}^\veps(\bd{y})
  \ud \bd{y} \ud \bd{p} \ud \bd{q},
\end{equation}
where $\Psi_0^\veps$ is the initial condition of
\eqref{eq:Schr}. Here the phase function $\Phi$ is given by
\begin{multline}
  \Phi(t, \bd{x}, \bd{y}, \bd{q}, \bd{p}) = S(t, \bd{q}, \bd{p}) +
  \frac{\I}{2} \abs{\bd{x} - \bd{Q}}^2 + \bd{P}\cdot(\bd{x} - \bd{Q})
  \\ + \frac{\I}{2} \abs{\bd{y} - \bd{q}}^2 - \bd{p}\cdot(\bd{y} -
  \bd{q}).
\end{multline}
The evolution of $(\bd{Q}, \bd{P})$ satisfies
\begin{equation}\label{Schr:Hflow}
  \begin{cases}
    \dsp\frac{\ud \bd{Q}}{\ud t} = \bd{P},
    \\[1em]
    \dsp\frac{\ud \bd{P}}{\ud t} = - \partial_{\bd{Q}}U,
  \end{cases}
\end{equation}
with initial conditions
\begin{equation}
  \bd{Q}(0, \bd{q}, \bd{p}) = \bd{q}, \quad
  \text{and} \quad \bd{P}(0, \bd{q}, \bd{p}) = \bd{p}.
\end{equation}
The action function $S(t,\bd{q}, \bd{p})$ satisfies
\begin{equation}\label{Schr:S}
  \frac{\ud S}{\ud t} = \frac{\abs{\bd{P}}^2}{2} - U(\bd{Q}),
\end{equation}
with initial condition $S(0,\bd{q}, \bd{p})=0$.  The evolution of
${a}(t,\bd{q}, \bd{p})$ satisfies
\begin{equation}\label{Schr:a}
  \frac{\ud a}{\ud t} = \frac{1}{2}a\tr\bigg(Z^{-1}
  \big(\partial_{\bd{z}}\bd{P}-\I\partial_{\bd{z}}\bd{Q}\partial_{\bd{Q}}^2U
  \big)\bigg),
\end{equation}
with initial condition ${a}(0,\bd{q}, \bd{p})=2^{d/2}$.

It is easy to see that \eqref{Schr:Hflow} and \eqref{Schr:S} are the
same as \eqref{eq:Hflow} and \eqref{eq:S} if one takes $ H_m(\bd{Q},
\bd{P})={\abs{\bd{P}}^2}/{2} + U(\bd{Q})$. The difference lies in
the amplitude evolution equation \eqref{Schr:a}. But it does not
raise any difficulty for Eulerian formulation.  One can still write
down Eulerian formulation based on Theorem~\ref{thm:deriv},
\begin{equation}
  {\Psi}^\veps_{\mathrm{EHK}}(t, \bd{x}) = \frac{1}{(2\pi \veps)^{3d/2}}
  \int {a}(t, \bd{Q}, \bd{P})e^{\I\Theta/\veps}
J^{-1} \ud \bd{P} \ud \bd{Q},
\end{equation}
where the phase function $\Theta$ is
\begin{equation}
  \Theta(t, \bd{x}, \bd{Q}, \bd{P})
  = S(t, \bd{Q}, \bd{P}) + \bd{P}\cdot(\bd{x} - \bd{Q})
  + \frac{\I}{2} \abs{\bd{x} - \bd{Q}}^2.
\end{equation}
Define the Liouville operator
\begin{equation*}
\LL=\partial_t+\bd{P}\cdot\partial_{\bd{Q}}
-\partial_{\bd{Q}}U\cdot\partial_{\bd{P}}.
\end{equation*}
Then the evolution of $S(t,\bd{Q}, \bd{P})$ satisfies
\begin{equation}
\LL S=\frac{\abs{\bd{P}}^2}{2}-U(\bd{Q}),
\end{equation}
with initial condition $S(0,\bd{Q}, \bd{P})=0$.  We introduce the
auxiliary function $\bd{\phi}(t,\bd{Q}, \bd{P})$, which satisfies
\begin{equation}
\LL \bd{\phi}=0,
\end{equation}
with initial condition
\begin{equation}
\bd{\phi}(0,\bd{Q}, \bd{P})=\bd{P}+\I \bd{Q}.
\end{equation}
With $\bd{\phi}$ determined, the evolution of $a(t,\bd{Q}, \bd{P})$ satisfies
\begin{equation}\label{Schr:AmpEu}
\begin{aligned}
\LL a=-\frac{1}{2}a\tr\bigg(Z^{-1}
  \big((
\partial_{\bd{Q}}\bd{\phi}
    )^{\TT}+\I ( \partial_{\bd{P}}\bd{\phi})^{\TT}
  \partial_{\bd{Q}}^2U
  \big)\bigg),
\end{aligned}
\end{equation}
where
\begin{equation}
Z=\bigl( \partial_{\bd{P}}\bd{\phi} \bigr)^{\TT}-\I \bigl(
\partial_{\bd{Q}}\bd{\phi}
    \bigr)^{\TT}.
\end{equation}

The initial condition of \eqref{Schr:AmpEu} is prepared as
\begin{equation}
a(0,\bd{Q}, \bd{P})=2^{d/2}\int
\Psi^\veps_{0}(\bd{y})\exp\bigg(\frac{\I}{\veps}\big(-\bd{P}
\cdot(\bd{y}-\bd{Q})+\frac{\I}{2}\abs{\bd{y}-\bd{Q}}^2\big)\bigg)
\ud\bd{y}.
\end{equation}

Therefore the numerical methods discussed in section
\ref{sec:algorithm} can also be applied to the Herman-Kluk propagator.

\section{Numerical examples}\label{sec:numerics}
In this section, we present four numerical examples to show the
performance of Eulerian frozen Gaussian approximation (EFGA) and
also Eulerian methods for the Herman-Kluk propagator. Two of the
examples correspond to EFGA of wave propagation discussed in Section
\ref{sec:WaveExa}, and the other corresponds to the Schr\"odinger
equation discussed in Section \ref{sec:SchrExa}. We consider WKB
initial conditions, and use the Eulerian method with local indicator
to compute wave propagation, and use the semi-Lagrangian method for
the Herman-Kluk propagator of the Schr\"odinger equation.

\subsection{Wave propagation}

\begin{example}[One-dimensional scalar wave equation]\label{exa:1}
\begin{equation*}
  \partial_t^2 u - c(x)^2 \partial_x^2 u = 0.
\end{equation*}
The wave speed is $c(x)=x^2$. The initial conditions are
\begin{align*}
&u_0=\exp\bigl(-100(x-0.5)^2\bigr)\exp\left(\frac{\I x}{\veps}\right),\\
&\partial_t
u_0=-\frac{\I x^2}{\veps}\exp\bigl(-100(x-0.5)^2\bigr)\exp\left(\frac{\I
x}{\veps}\right).
\end{align*}
\end{example}

The final time is $T=0.8$. We plot the real part of the wave field
obtained by EFGA compared with the true solution in
Figure~\ref{fig:ex1_real} for $\veps=1/64,\;1/128,\;1/256$. As one
can see, the span of the solution reaches $2$ roughly at $T=0.8$,
although it starts with only $0.5$ approximately. Apparently the
wave spreads quickly in this example. Table~\ref{tab:ex1_err} shows
the $\ell^\infty$ and $\ell^2$ errors of the EFGA solution. The
convergence orders in $\veps$ of $\ell^\infty$ and $\ell^2$ norms
are $1.02$ and $1.20$ separately for EFGA, which confirms the
asymptotic accuracy.  The true solution is computed by the finite
difference method using the mesh size of $N_x=49152$ and the time
step of $N_t=524288$ for domain $[0,6]$. We take $N_t=1024$, $\delta
q=\delta p=\delta y=1/128$ and $\delta x=1/2048$ in EFGA.

\begin{table}
\caption{Example \ref{exa:1}, the $\ell^\infty$ and
  $\ell^2$ errors for EFGA.}\label{tab:ex1_err}
\begin{tabular}{cccc}
\hline \\[-1em] $\veps$ &
${1}/{2^6}$ & ${1}/{2^7}$ &
${1}/{2^8}$ \\ \hline \\[-1em]
$\norm{u-u^{\EFGA}}_{\ell^\infty}$ & $1.46\times 10^{-1}$  &
$6.39\times10^{-2}$ &
$3.52\times10^{-2}$\\
\hline
\\[-1em] $\norm{u-u^{\EFGA}}_{\ell^2}$& $4.81\times 10^{-2}$ & $2.39\times10^{-2}$ & $9.11\times10^{-3}$ \\
\hline
\end{tabular}
\end{table}

\begin{figure}[h t p]
\begin{tabular}{cc}
\resizebox{2.3in}{!}{\includegraphics{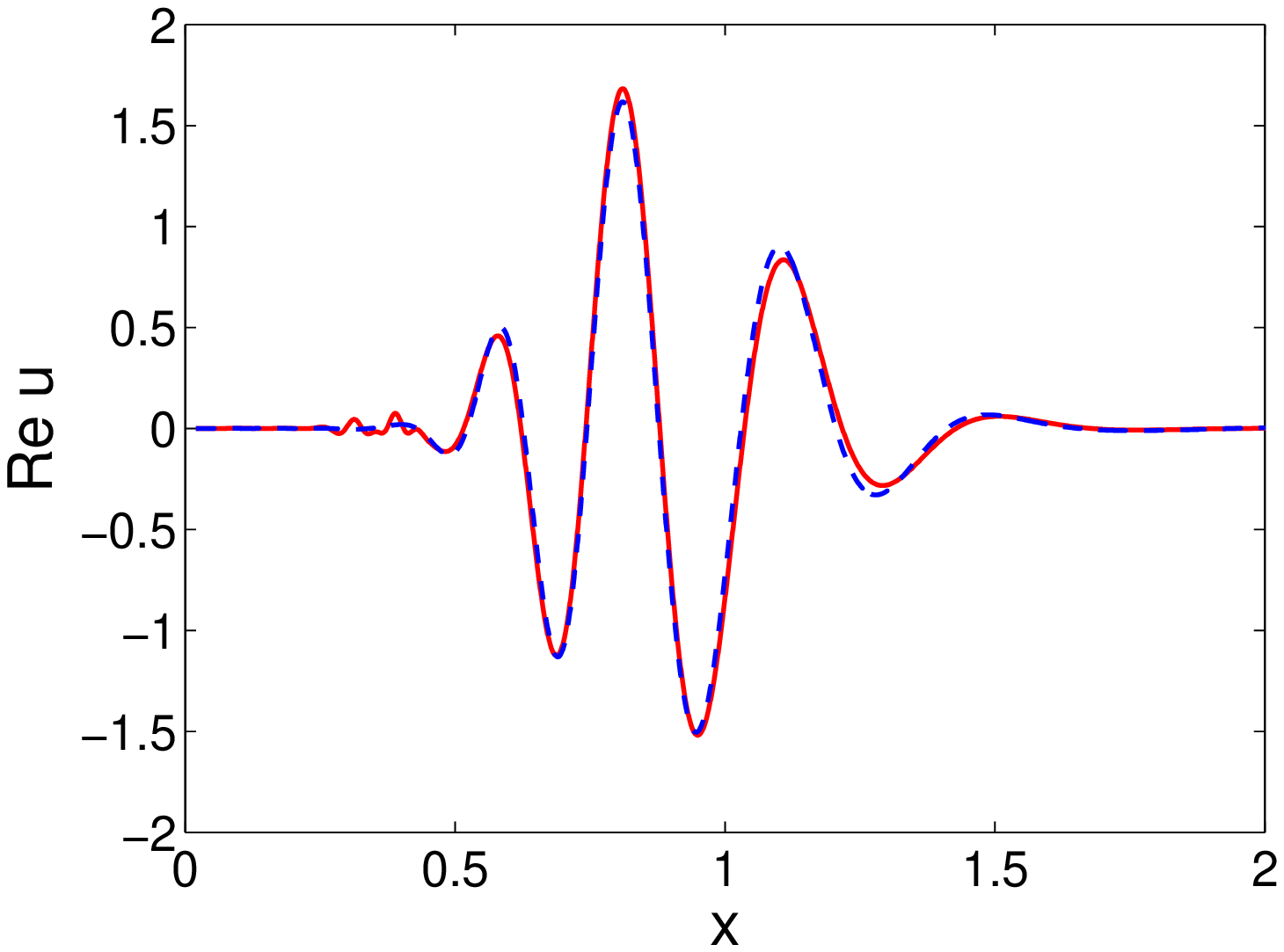}} &
\resizebox{2.3in}{!}{\includegraphics{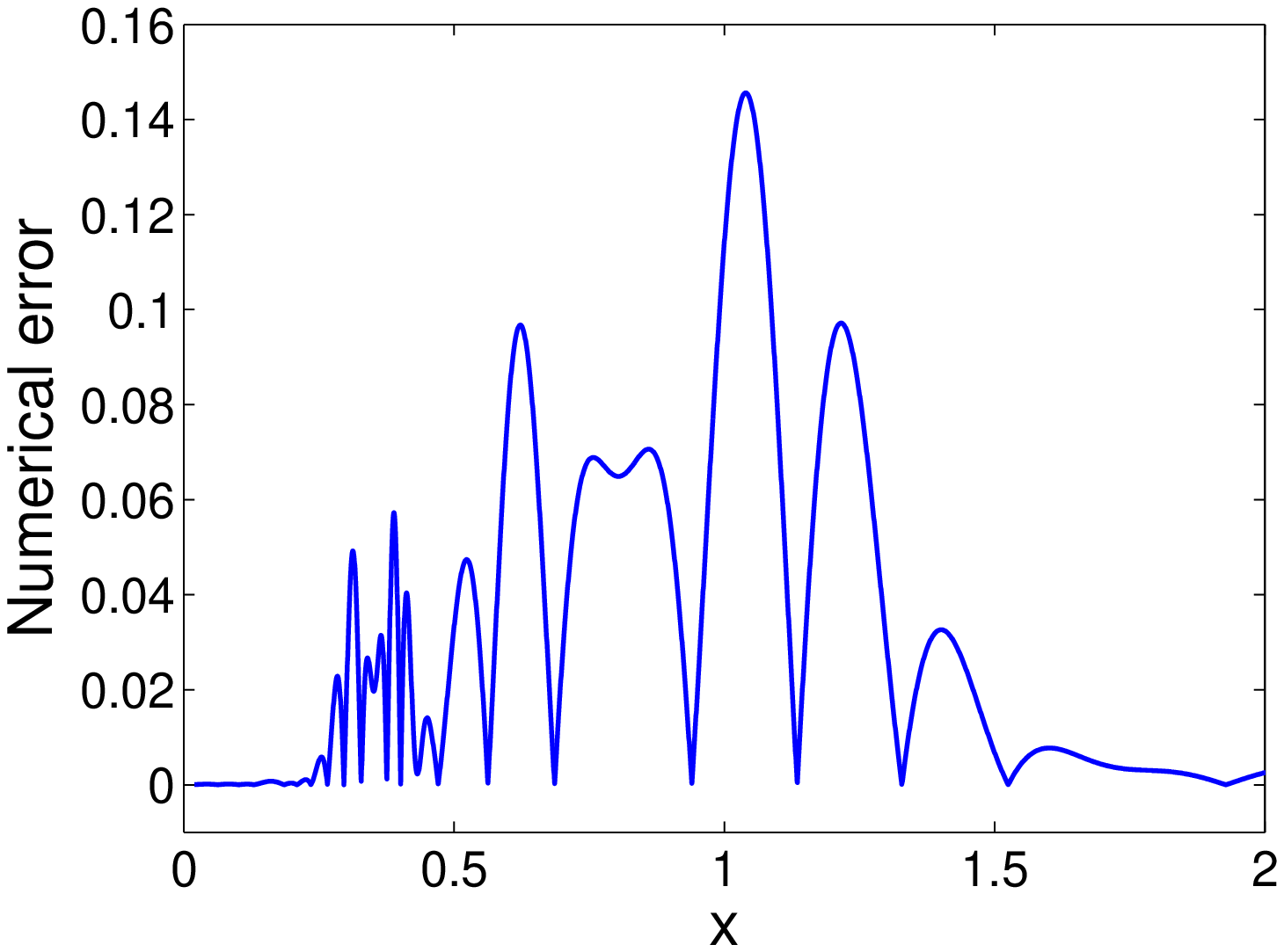}} \\
  \multicolumn{2}{c}{(a) $\veps=\frac{1}{64}$} \\[3mm]
\resizebox{2.3in}{!}{\includegraphics{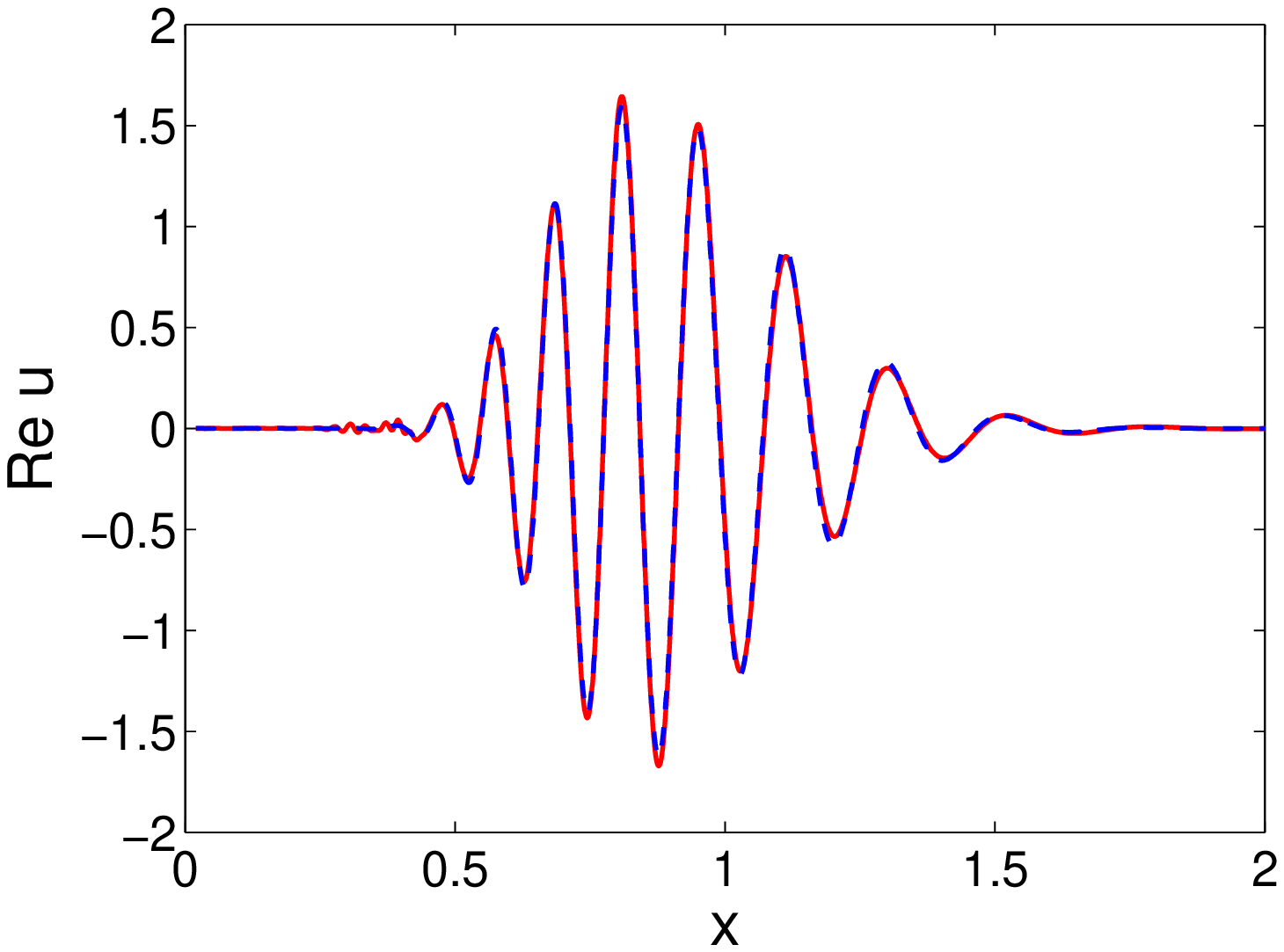}} &
\resizebox{2.3in}{!}{\includegraphics{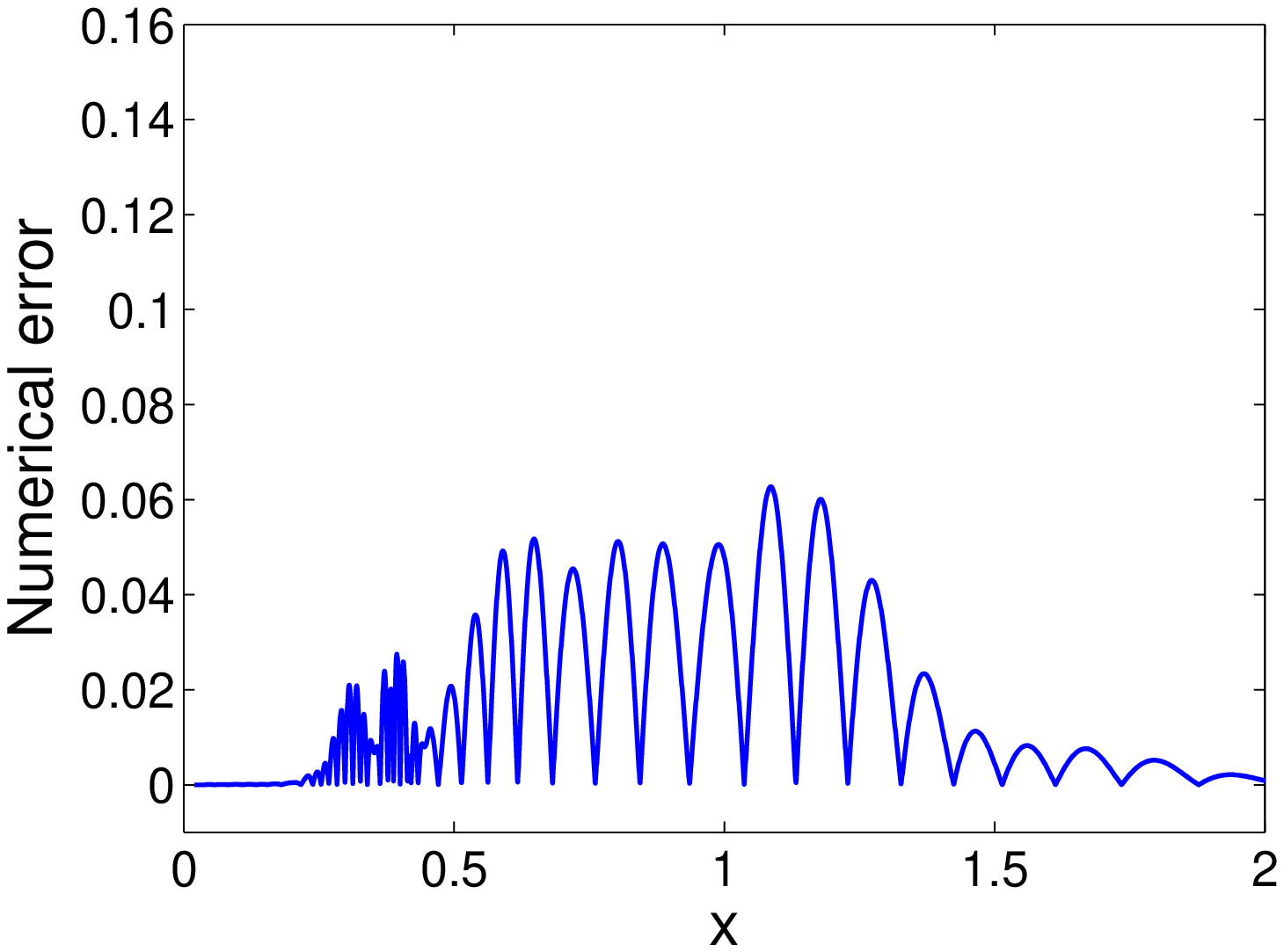}} \\
  \multicolumn{2}{c}{(b) $\veps=\frac{1}{128}$} \\[3mm]
\resizebox{2.3in}{!}{\includegraphics{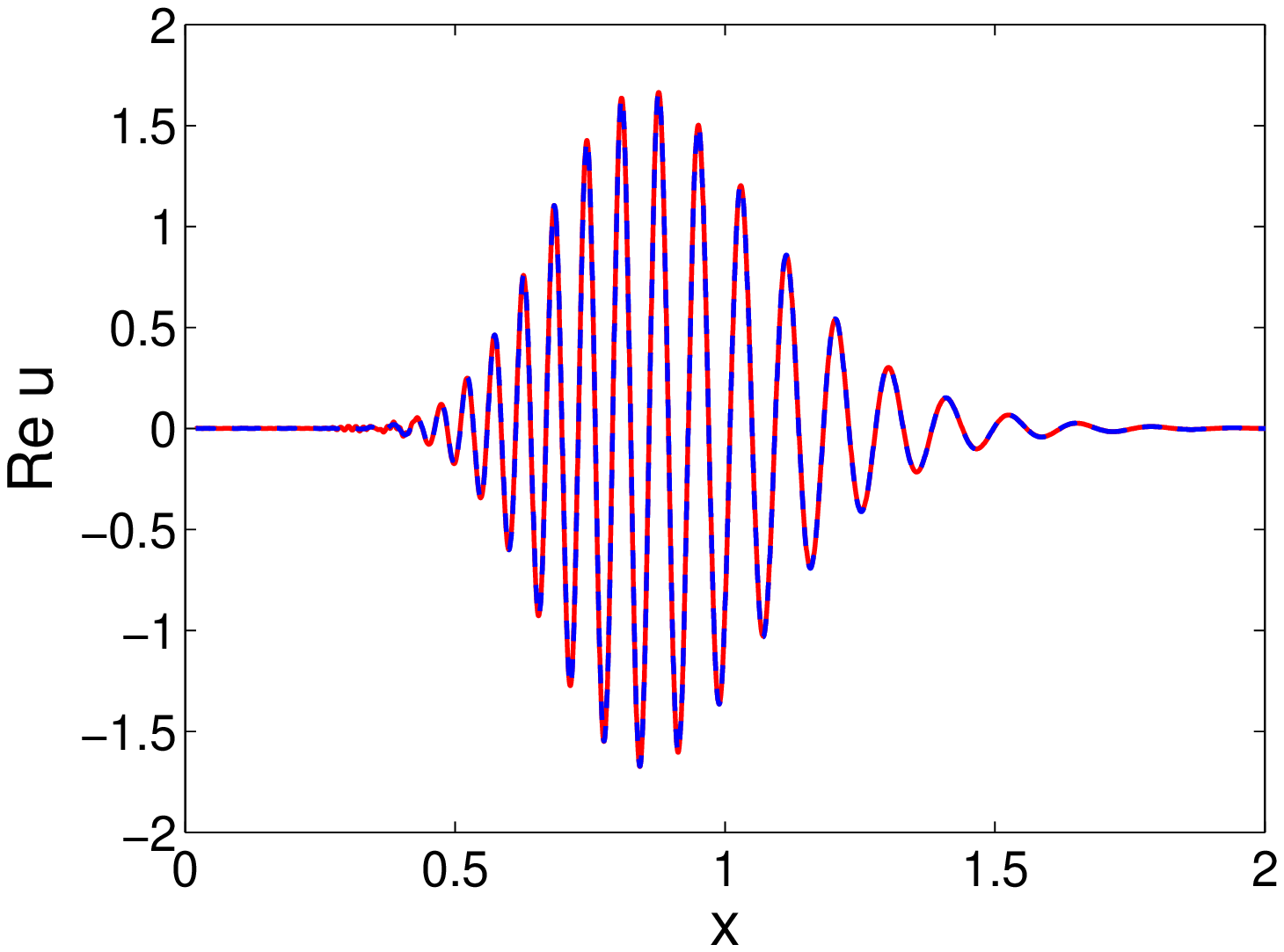}} &
\resizebox{2.3in}{!}{\includegraphics{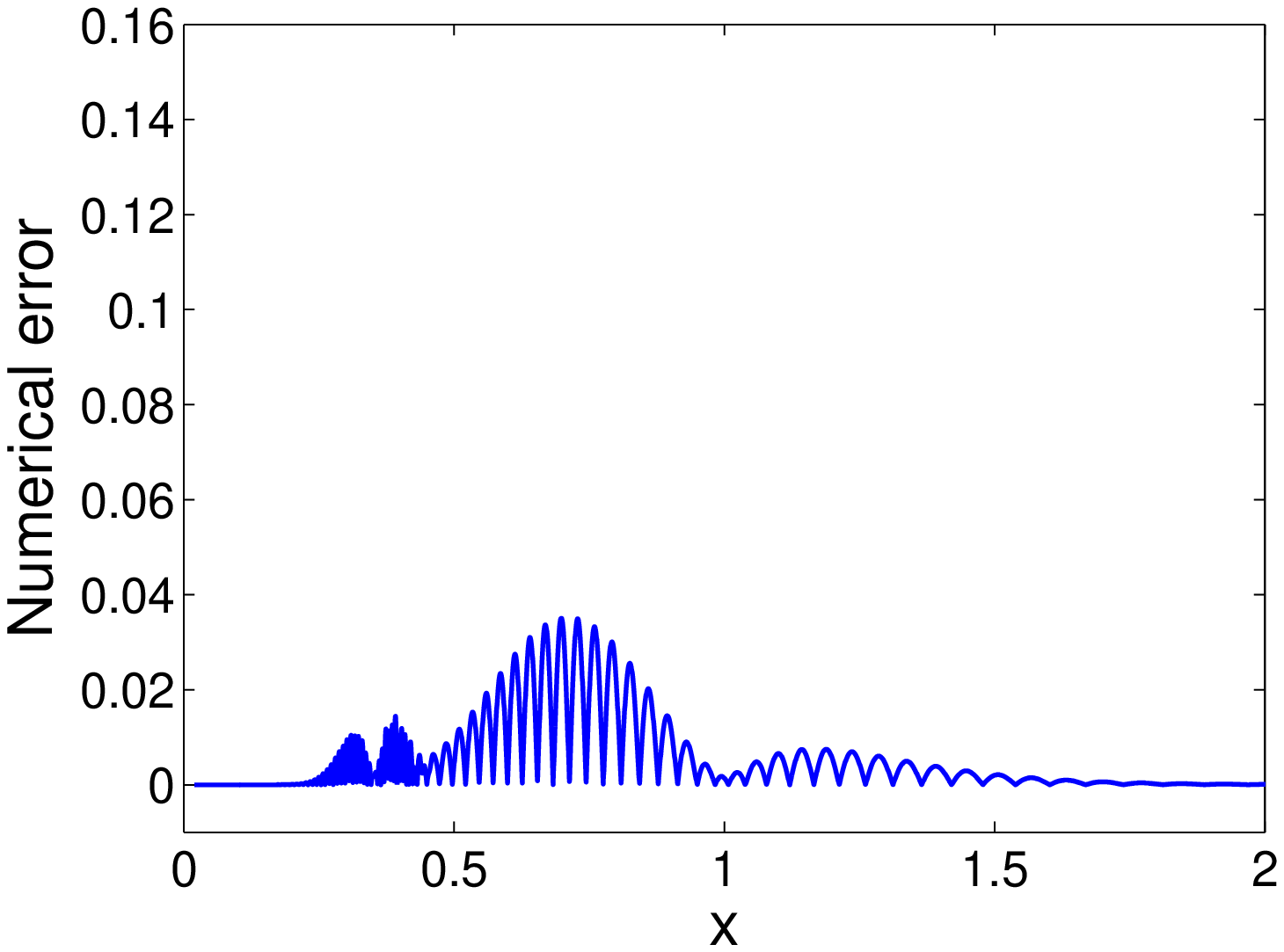}} \\
  \multicolumn{2}{c}{(c) $\veps=\frac{1}{256}$}
\end{tabular}
\caption{Example \ref{exa:1}, the comparison of the true solution
(solid line) and the solution by EFGA (dash line). Left: the real
part of wave field; right: the errors between
them.}\label{fig:ex1_real}
\end{figure}

\begin{example}[Two-dimensional acoustic wave equations] \label{exa:2}
\begin{equation*}
  \begin{cases}
    \partial_t \bd{V} +  \nabla \Pi = 0; \\
    \partial_t \Pi + c^2(\bd{x}) \nabla\cdot \bd{V} = 0,
  \end{cases}
\end{equation*}
where $\bd{V}=(V_1,V_2)$, $\bd{x}=(x_1,x_2)$ and $c(\bd{x})=1$. The
initial conditions are
\begin{align*}
  & \Pi_0=\sqrt{1+\sin^2(4x_2)/16} \exp\bigl(-100(x_1^2+x_2^2)\bigr)
  \exp\left(\frac{\I}{\veps}\big(-x_1+\cos(4x_2)/16\big)\right),\\
  & V_{1,0}=-
  \exp\bigl(-100(x_1^2+x_2^2)\bigr)
  \exp\left(\frac{\I}{\veps}\big(-x_1+\cos(4x_2)/16\big)\right),\\
  & V_{2,0}=-\frac{\sin(4x_2)}{4}\exp\bigl(-100(x_1^2+x_2^2)\bigr)
  \exp\left(\frac{\I}{\veps}\big(-x_1+\cos(4x_2)/16\big)\right).
\end{align*}
\end{example}

The final time is $T=1.0$. We take $\veps=1/64$.
Figure~\ref{fig:ex2_pres} compares the pressure $\Pi$ of the true
solution with the one by EFGA. Figure~\ref{fig:ex2_vel} compares the
velocity $\bd{V}$ of the true solution with the one by EFGA. It is
clear that EFGA can provide a good approximation to both the
pressure and velocity for acoustic wave propagation in two
dimension. The true solution is given by the spectral method using
the mesh $\delta x_1=\delta x_2=1/512$ for domain
$[-1.5,0.5]\times[-1,1]$. We take $\delta q_1=\delta q_2=\delta
p_1=\delta p_2=\delta y_1=\delta y_2=1/32$ and $\delta x_1=\delta
x_2=1/64$ in EFGA.

\begin{figure}[h t p]
\begin{tabular}{ccc}
  \resizebox{2.3in}{!}{\includegraphics{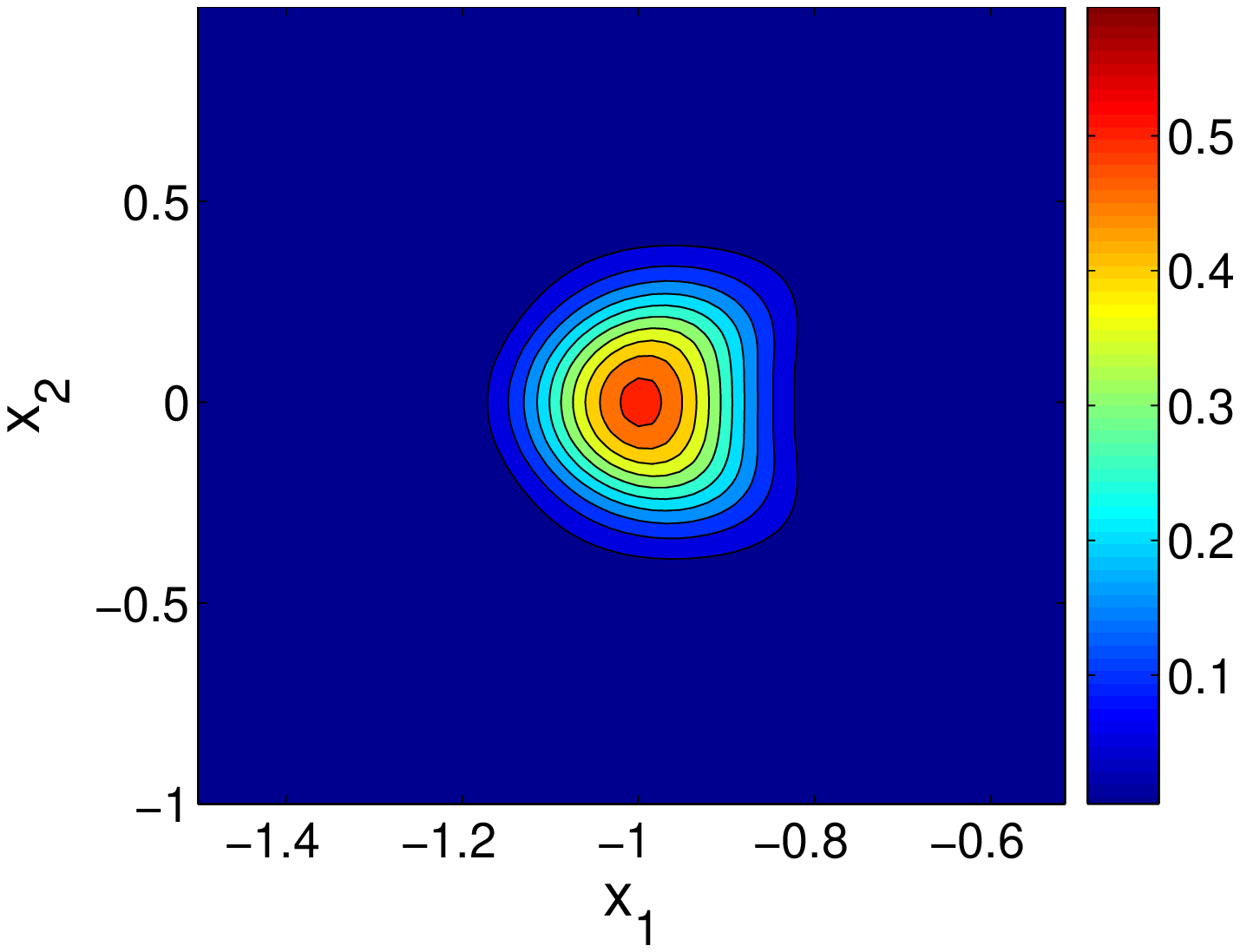}} &
  \resizebox{2.3in}{!}{\includegraphics{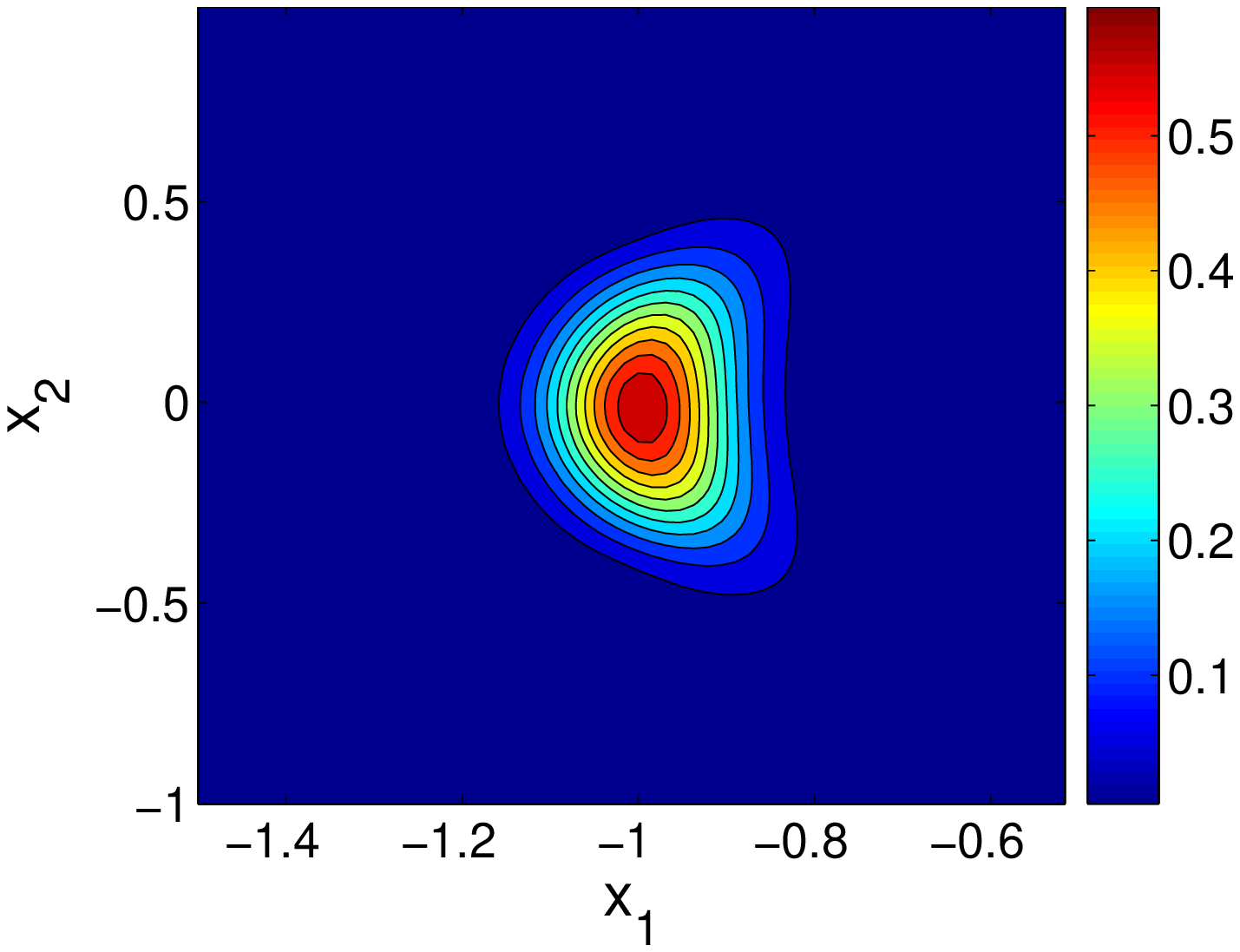}}
\end{tabular}
\begin{center}
  \resizebox{2.3in}{!}{\includegraphics{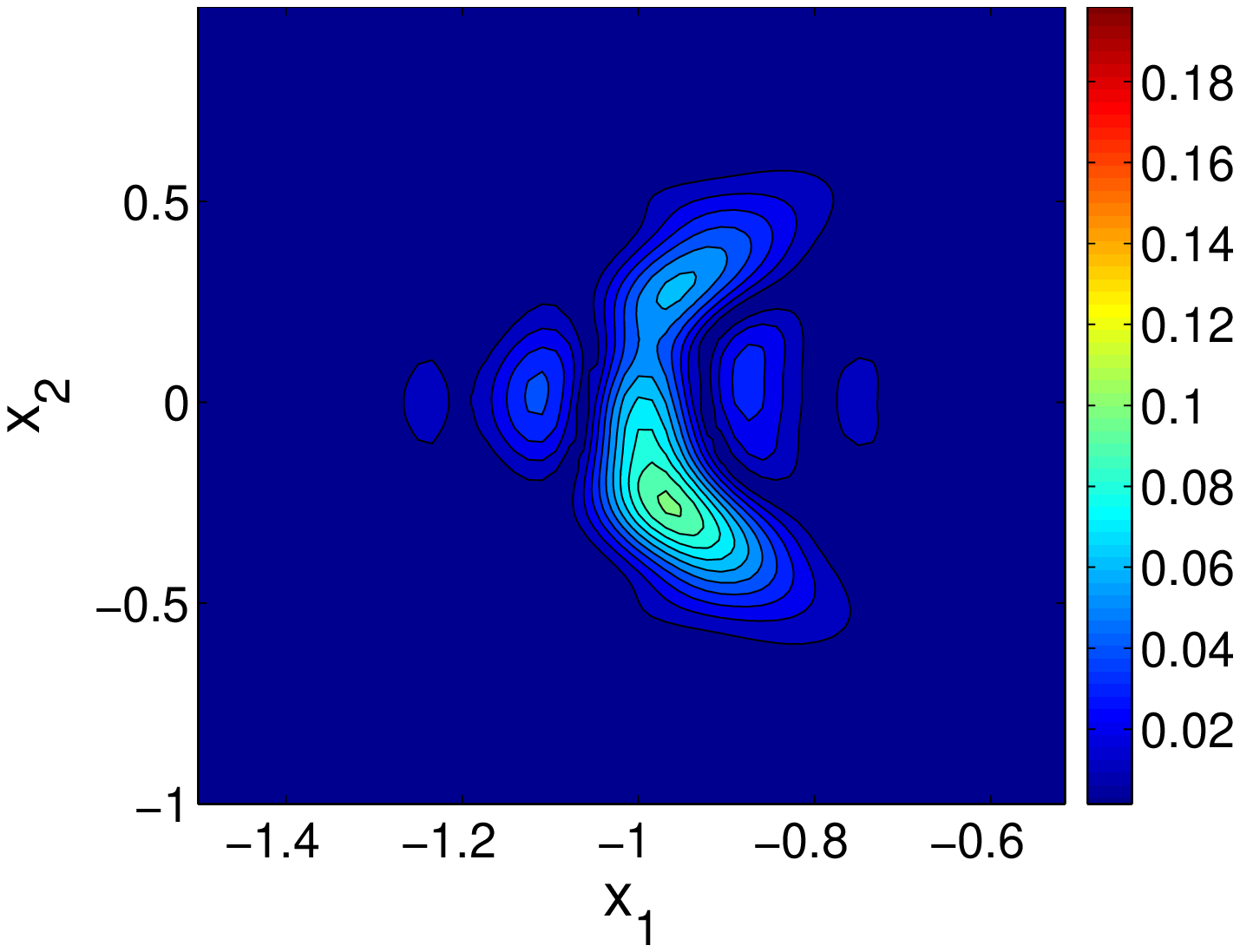}}
\end{center}
\caption{Example \ref{exa:2}, the comparison of the true solution
and the solution by EFGA. Top (left): pressure of EFGA; top (right):
pressure of true solution; bottom: the error between
them.}\label{fig:ex2_pres}
\end{figure}

\begin{figure}[h t p]
\begin{tabular}{cc}
  \resizebox{2.3in}{!}{\includegraphics{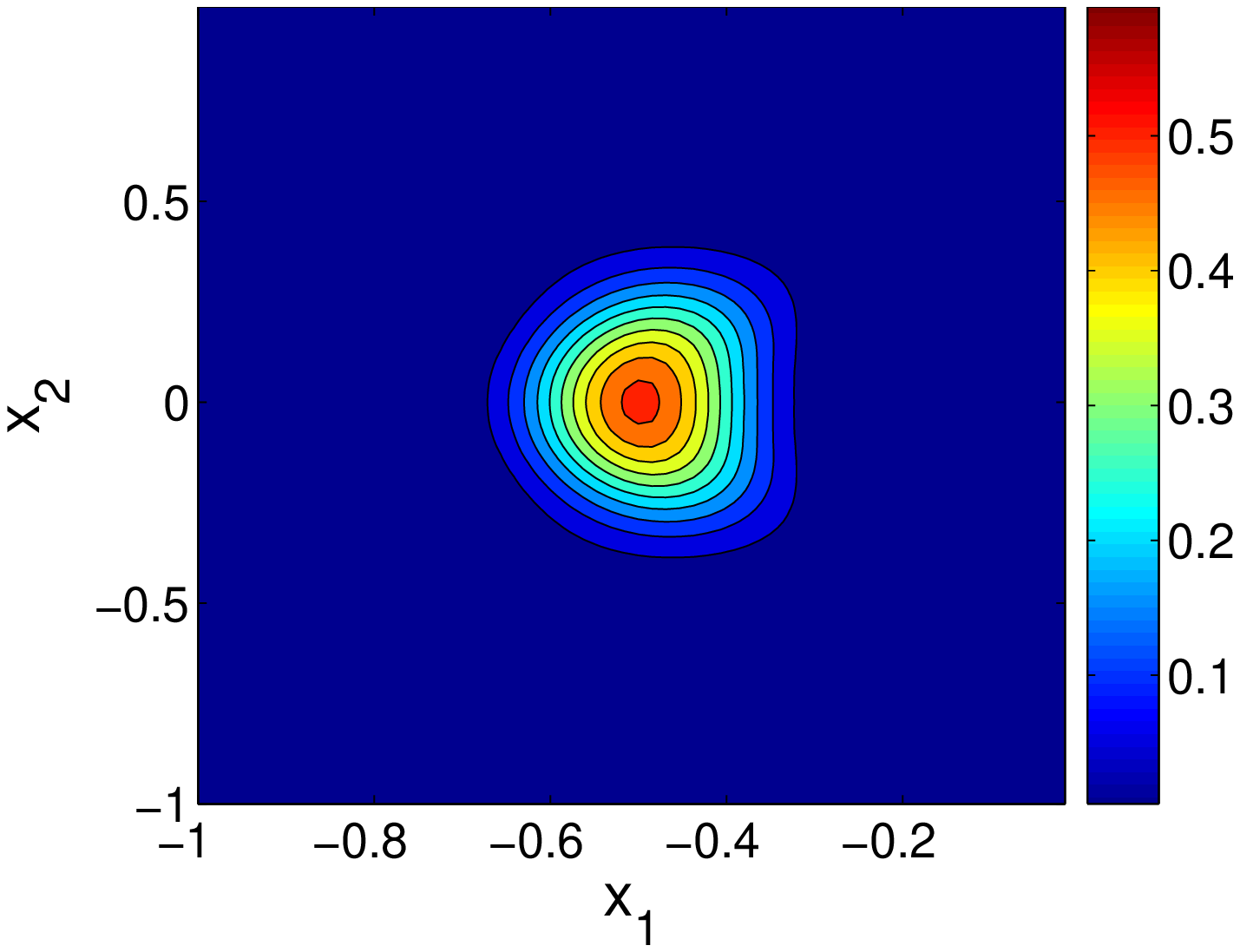}} &
  \resizebox{2.3in}{!}{\includegraphics{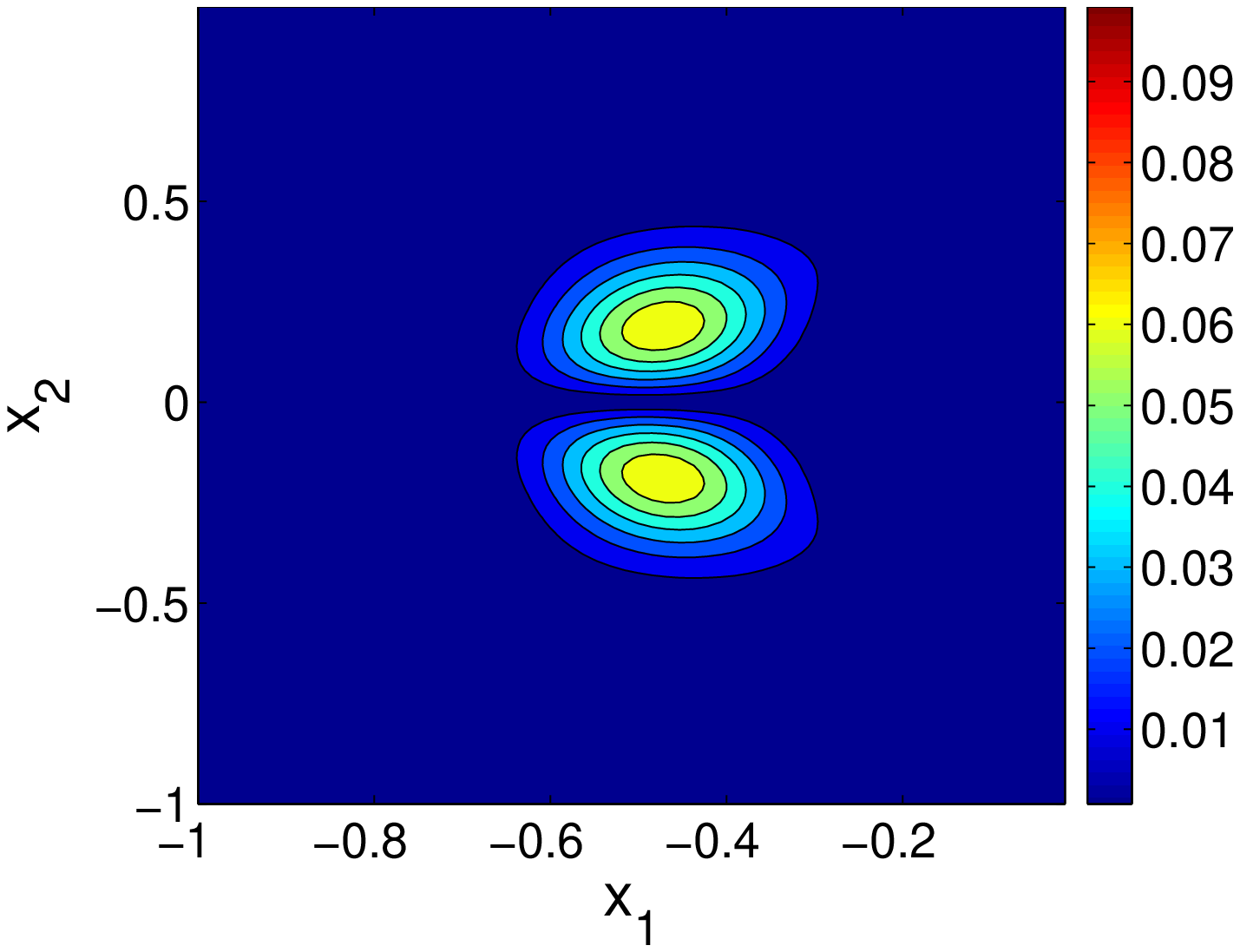}} \\
  \multicolumn{2}{c}{(a) Eulerian Frozen Gaussian approximation} \\[3mm]
  \resizebox{2.3in}{!}{\includegraphics{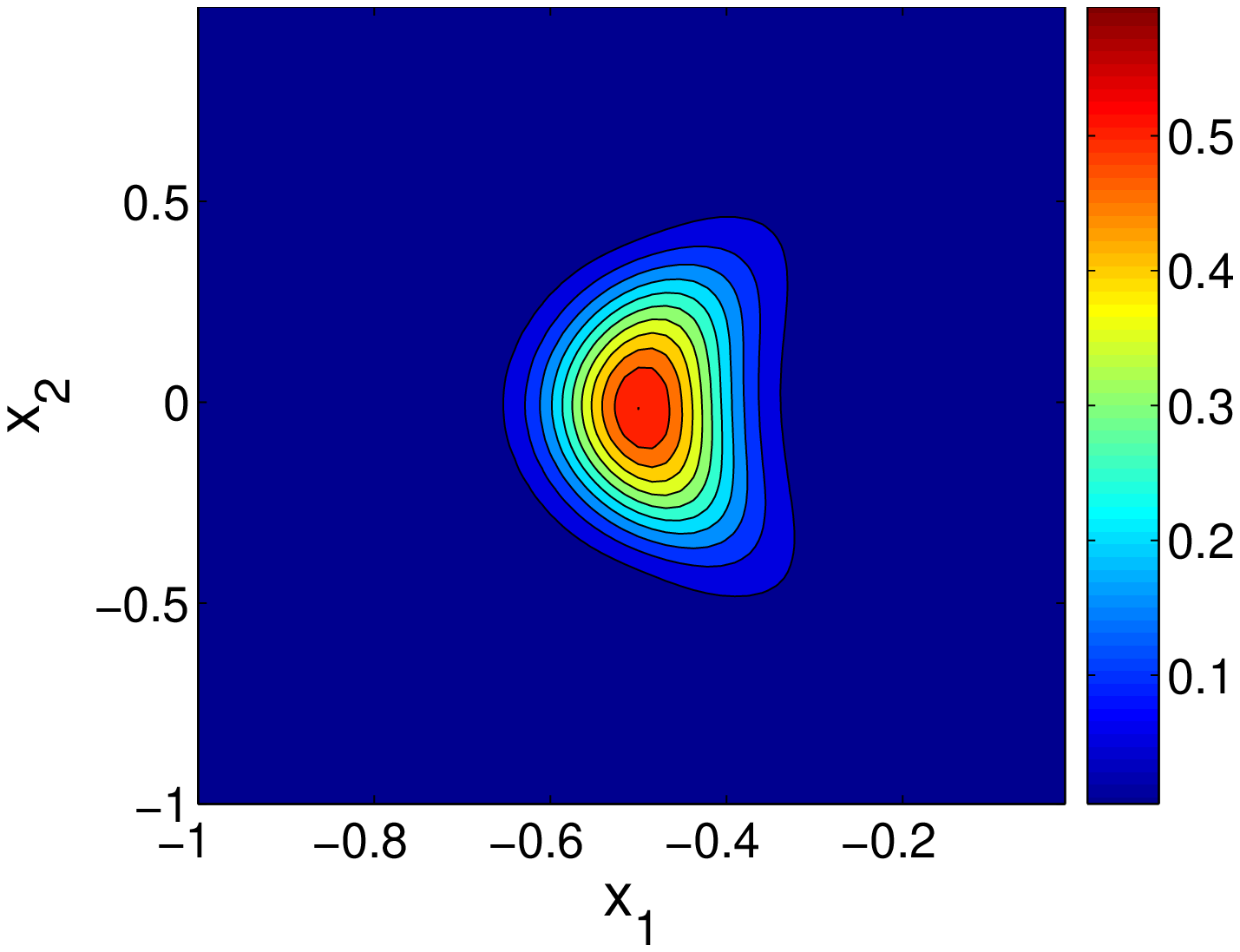}} &
  \resizebox{2.3in}{!}{\includegraphics{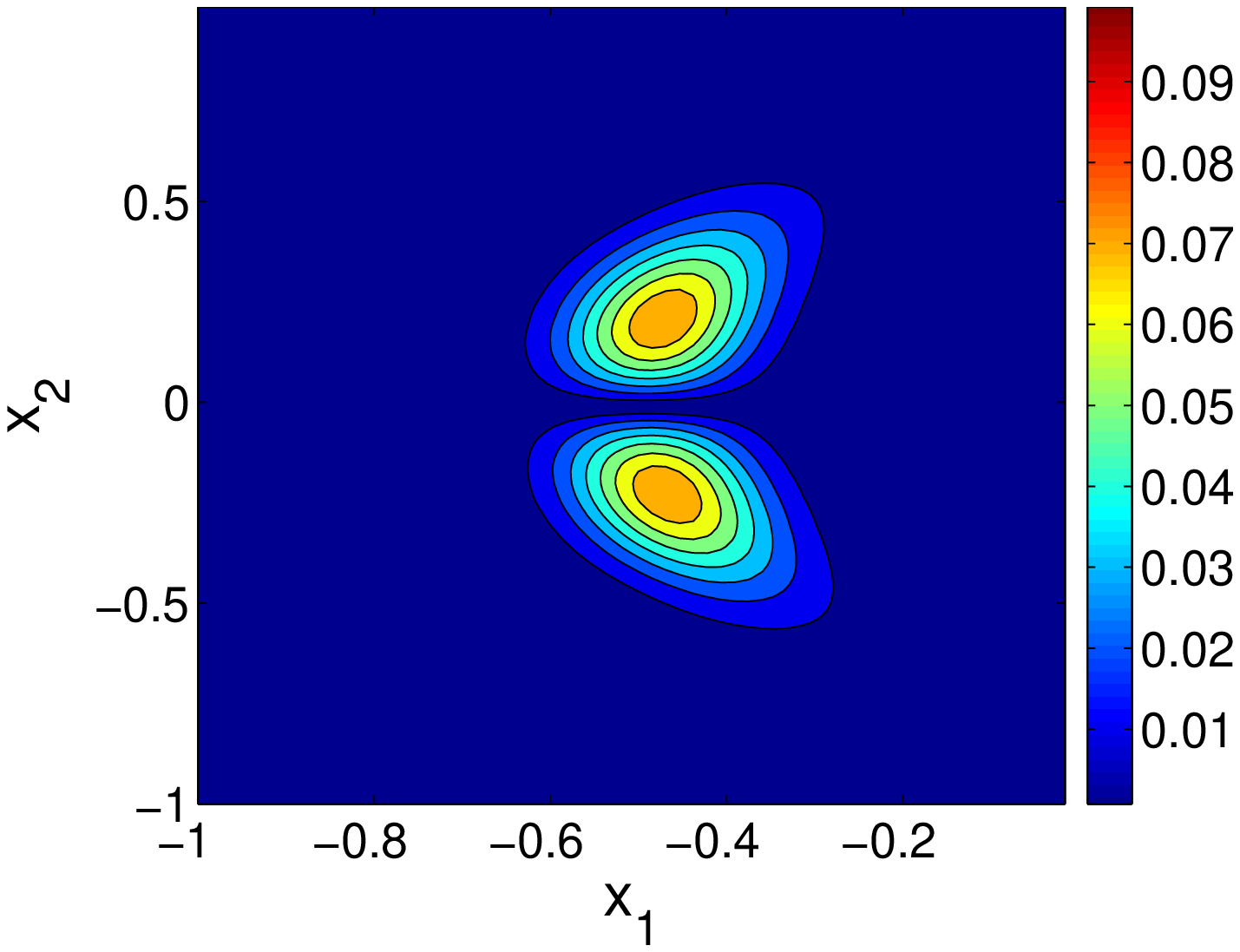}} \\
  \multicolumn{2}{c}{(b) True solution} \\[3mm]
  \resizebox{2.3in}{!}{\includegraphics{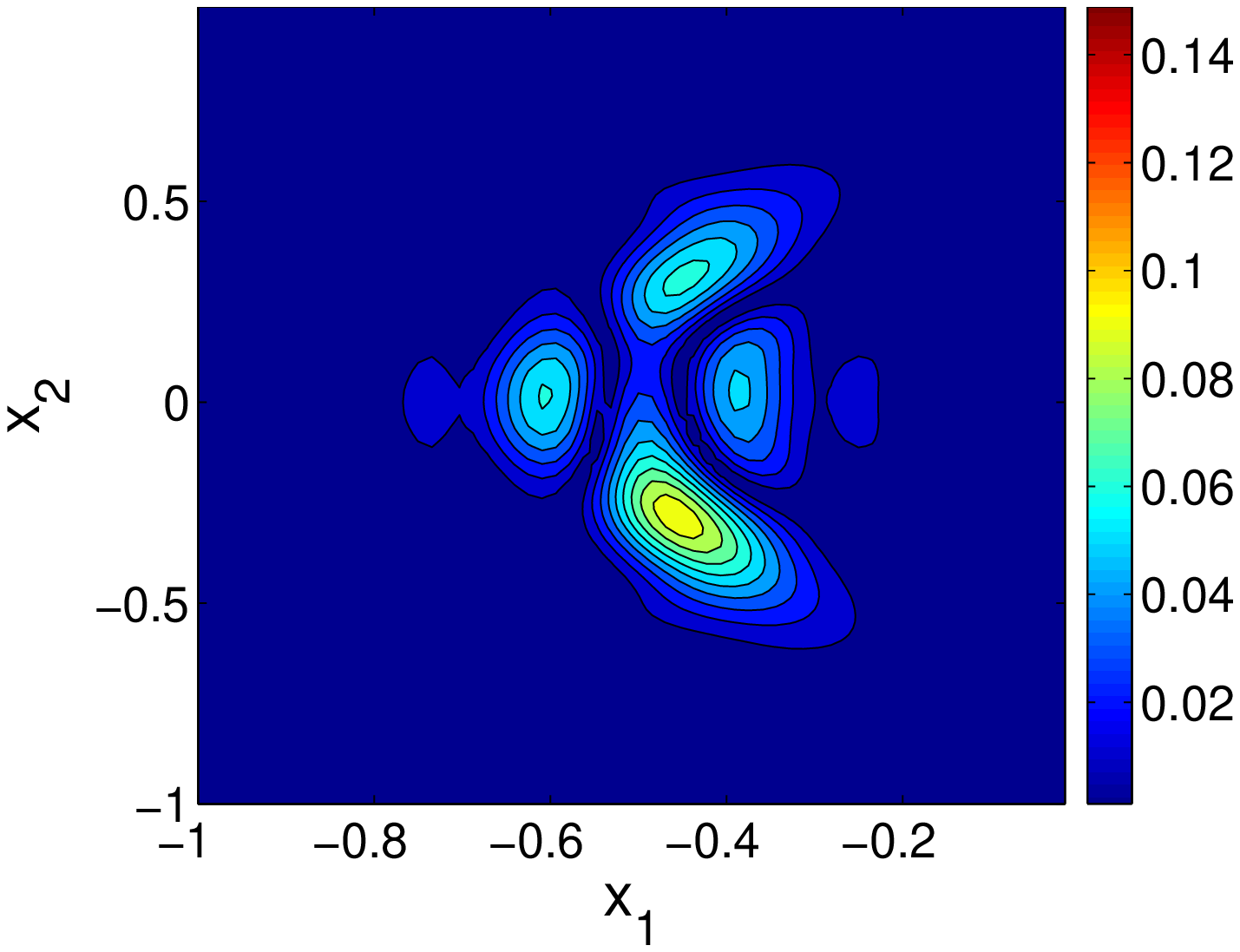}} &
  \resizebox{2.3in}{!}{\includegraphics{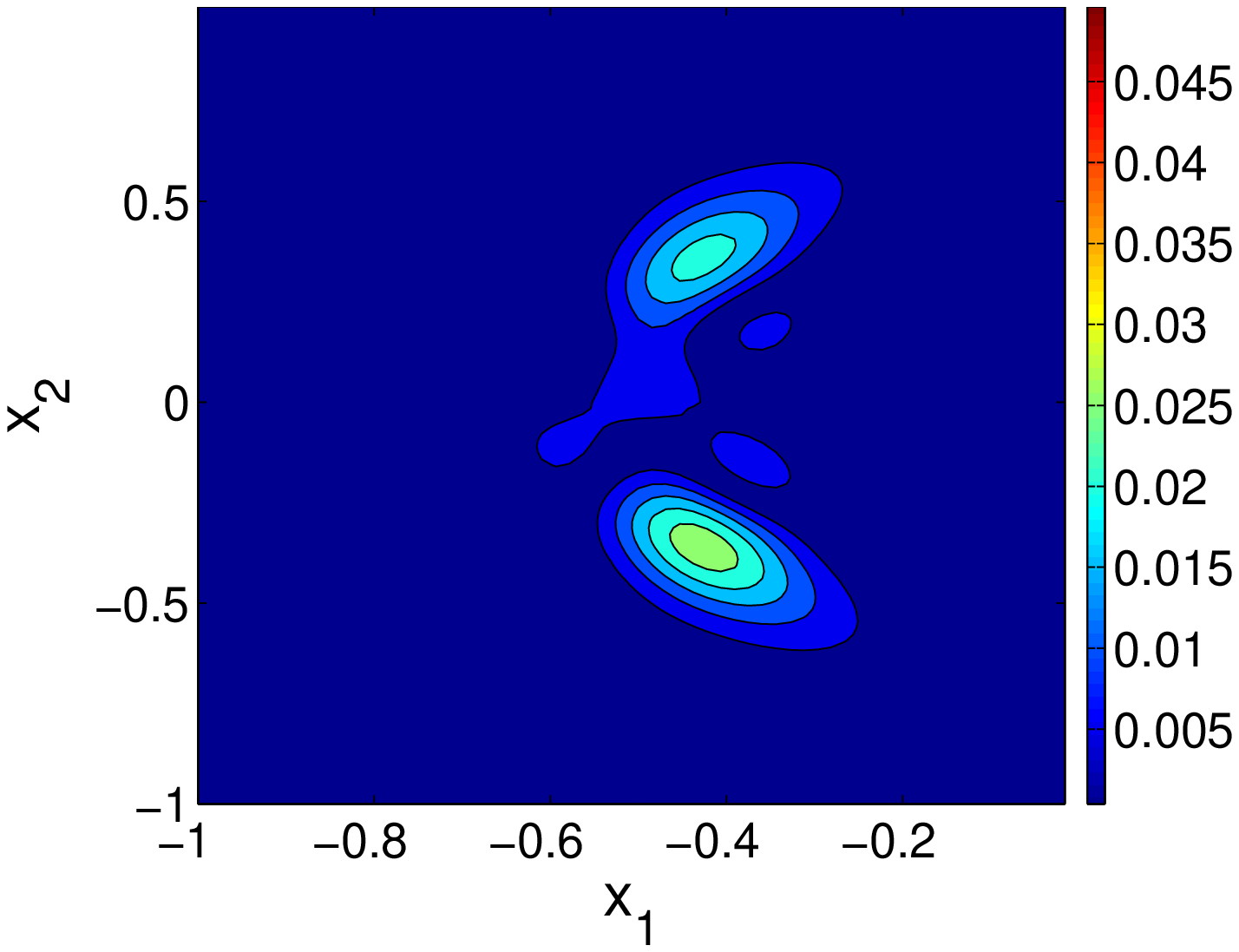}} \\
  \multicolumn{2}{c}{(c) Errors}
\end{tabular}
\caption{Example \ref{exa:2}, the comparison of the true solution
and the solution by EFGA. Left: velocity component $V_1$;
  right: velocity component $V_2$.}\label{fig:ex2_vel}
\end{figure}

\subsection{Schr\"odinger equation}

\begin{example}[One-dimensional Schr\"odinger equation] \label{exa:4}
\begin{equation*}
i\veps\frac{\partial \Psi^\veps}{\partial
t}=-\frac{\veps^2}{2}\Delta\Psi^\veps, \quad {x}\in\mathbb{R} \,,
\end{equation*}
and the initial condition is
\[\Psi_0^\veps=\exp\left(-\I\frac{x}{\veps}-\frac{x^2}{2\veps}\right).\]
\end{example}

We use this one-dimensional Schr\"odinger equation with zero
potential as an example to compare the performance of Lagrangian and
Eulerian methods. The true solution can be given analytically,
\begin{equation*}
\Psi^\veps(t,x)=\frac{1}{1+t\I}\exp
\biggl(\frac{\I}{\veps}\Bigl(x-\frac{t}{2}+\frac{(t+\I)(x-t)^2}{2(1+t^2)}\Bigr)\biggr),
\end{equation*}
which implies the solution spreads as time increases.

We choose $\veps=1/128$ and evolve the equation up to $T=10$. The
mesh sizes are $\delta x=\delta y=\delta q=\delta p =1/64$. We take
$N_q=64$ and $N_p=33$ in the Lagrangian method. The comparison of
wave amplitudes and numerical errors are presented in
Figure~\ref{fig:ex4_amp}. One can see that, when the divergence of
particle trajectories occurs, Eulerian method has a much better
resolution than the Lagrangian method.

\begin{figure}[h t p]
\begin{tabular}{cc}
  \resizebox{2.3in}{!}{\includegraphics{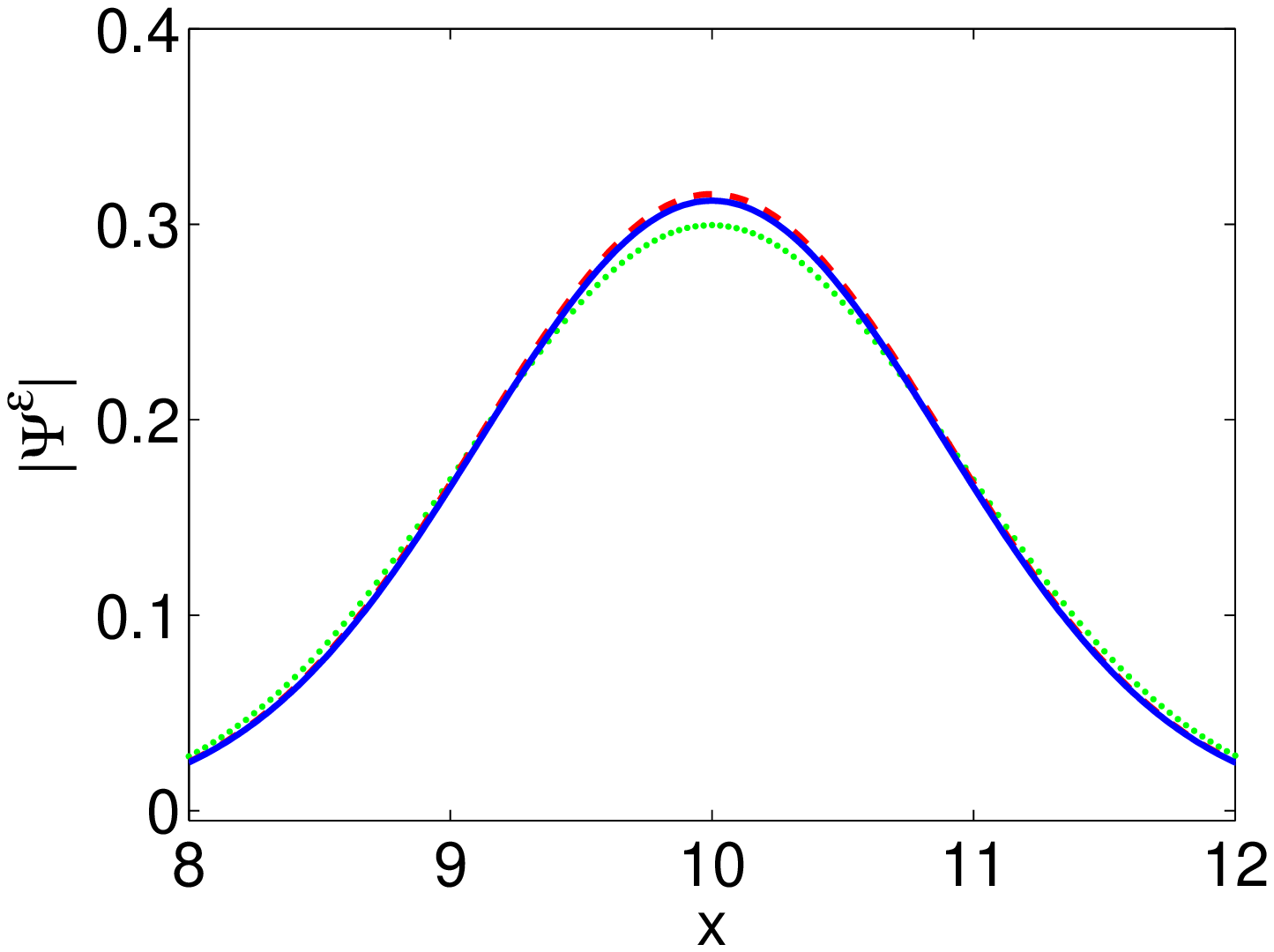}} &
  \resizebox{2.3in}{!}{\includegraphics{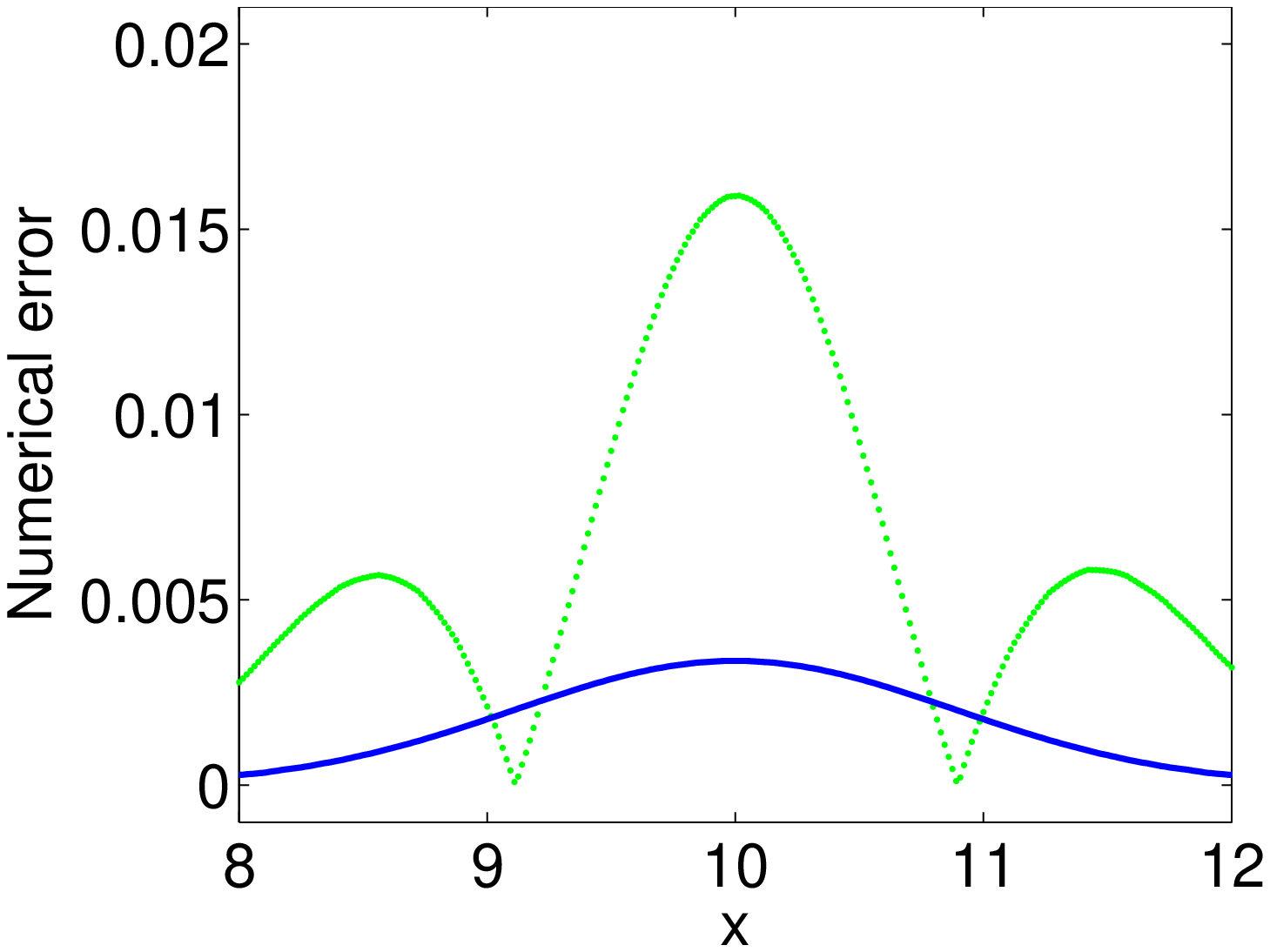}}
\end{tabular}
\caption{Example \ref{exa:4}. Left: the wave amplitude comparison of
the true solution (dashed line), the Lagrangian method (dots) and
the Eulerian method (solid line); right: the numerical errors of the
Lagrangian method (dots) and the Eulerian method (solid
line).}\label{fig:ex4_amp}
\end{figure}

\begin{example}[Two-dimensional Schr\"odinger equation] \label{exa:3}
\begin{equation*}
i\veps\frac{\partial \Psi^\veps}{\partial
t}=-\frac{\veps^2}{2}\Delta\Psi^\veps+\frac{1}{2}\abs{\bd{x}}^2\Psi^\veps,
\quad \bd{x}\in\mathbb{R}^2 \,,
\end{equation*}
and the initial condition is
\[\Psi_0^\veps=\exp\bigl(-25(x_1^2+x_2^2)\bigr)
  \exp\left(\frac{\I}{2\veps}\sin(x_1)\sin(x_2)\right).\]
\end{example}

This is an example of Schr\"odinger equation in two dimension, which
describes the dynamics of electron under harmonic potential.  We
take $\veps=1/128$. Figure~\ref{fig:ex3_amp} compares the wave
amplitude of the true solution with the numerical one at time
$T=0.5$ and $T=1$. This shows that Eulerian Herman-Kluk propagator
has good performances in both cases of solution spreading and
localizing. The true solution is given by the spectral method using
the mesh $\delta x_1=\delta x_2=1/512$ for domain
$[-1,1]\times[-1,1]$. In the numerical approximation, the mesh sizes
are chosen to be $\delta q_1=\delta q_2=\delta p_1=\delta p_2=\delta
y_1=\delta y_2=1/32$ in discretization of integrals and $\delta
x_1=\delta x_2=1/32$ in reconstruction of solution.

\begin{figure}[h t p]
\begin{tabular}{cc}
  \resizebox{2.3in}{!}{\includegraphics{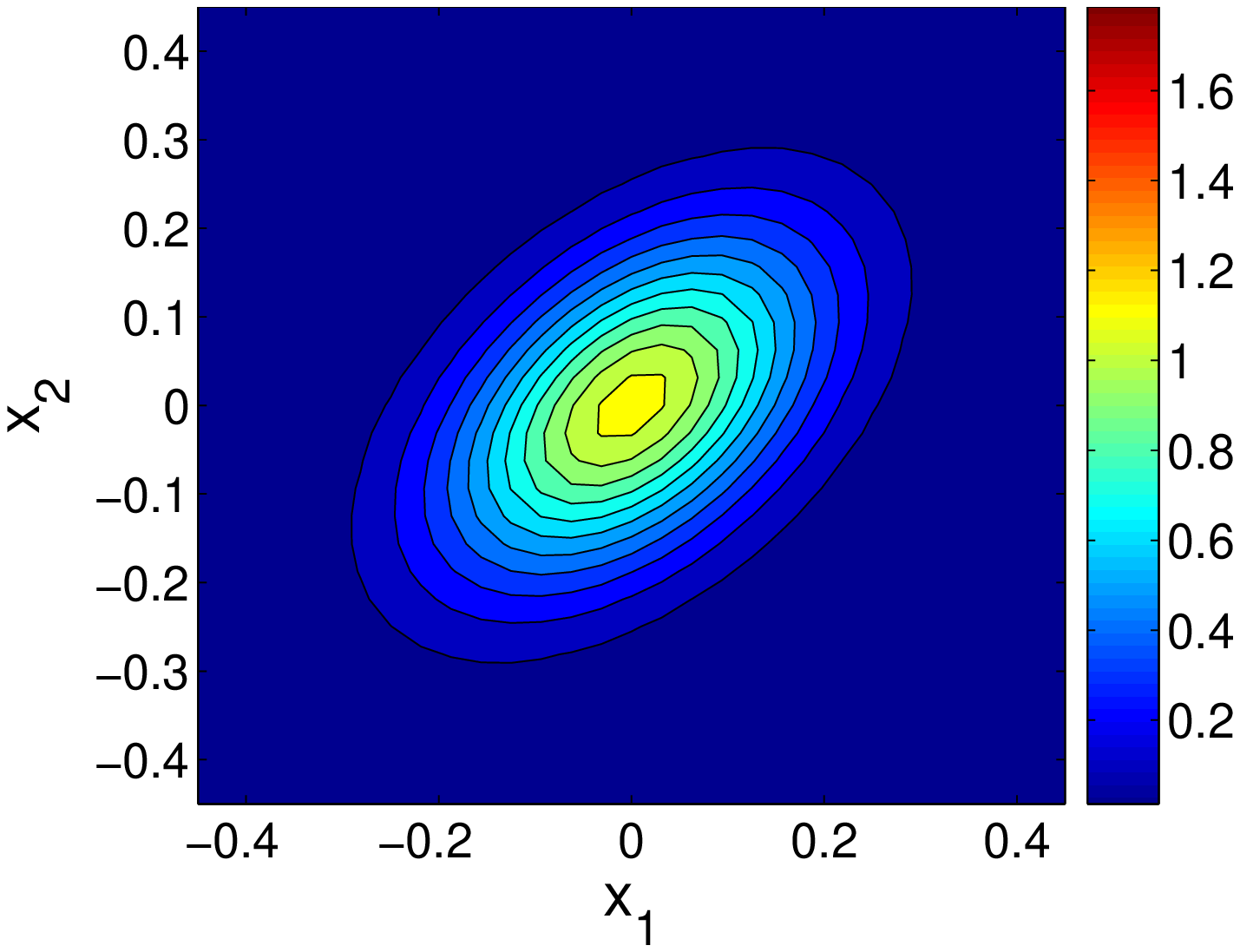}} &
  \resizebox{2.3in}{!}{\includegraphics{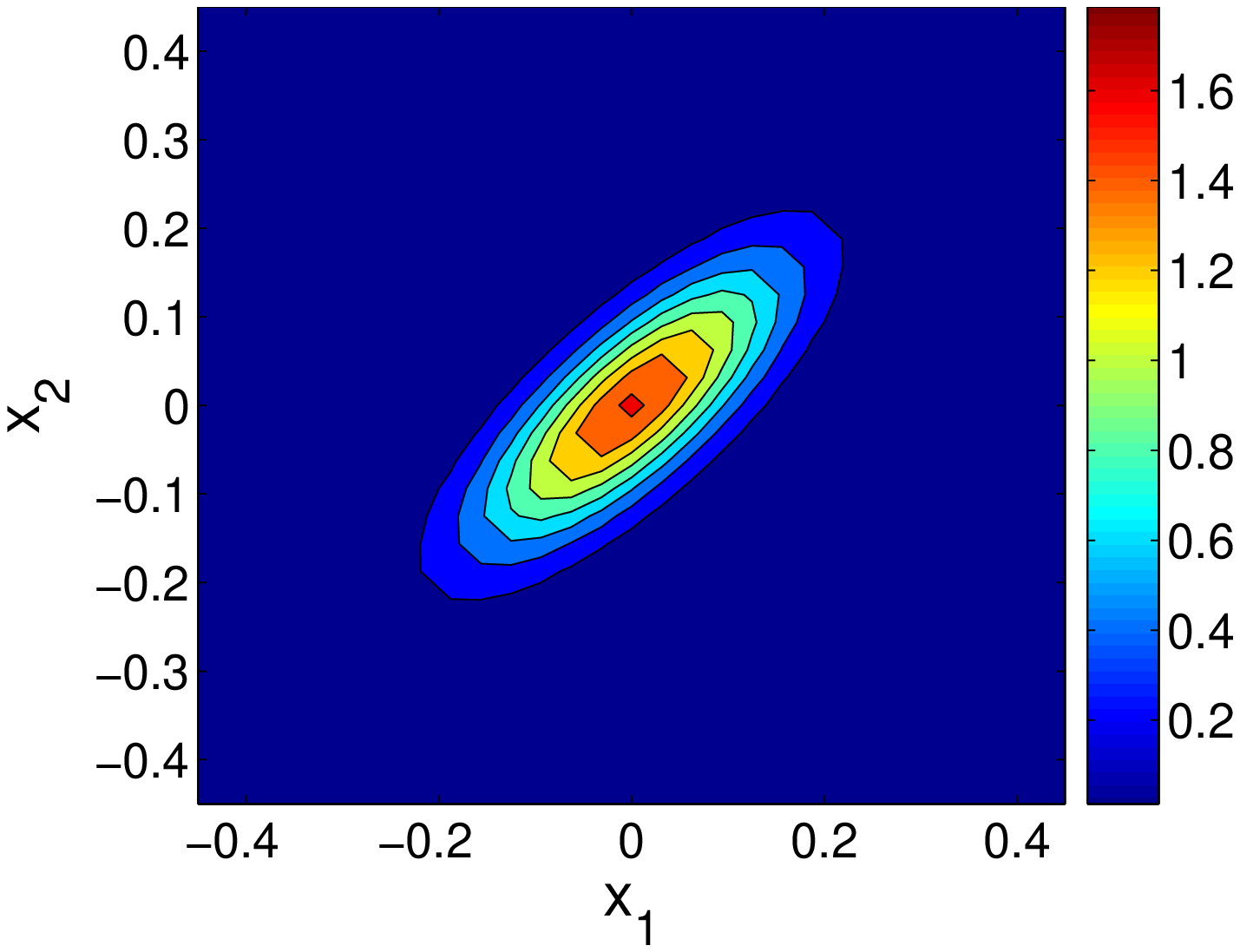}} \\
  \multicolumn{2}{c}{(a) Eulerian Herman-Kluk propagator} \\[3mm]
  \resizebox{2.3in}{!}{\includegraphics{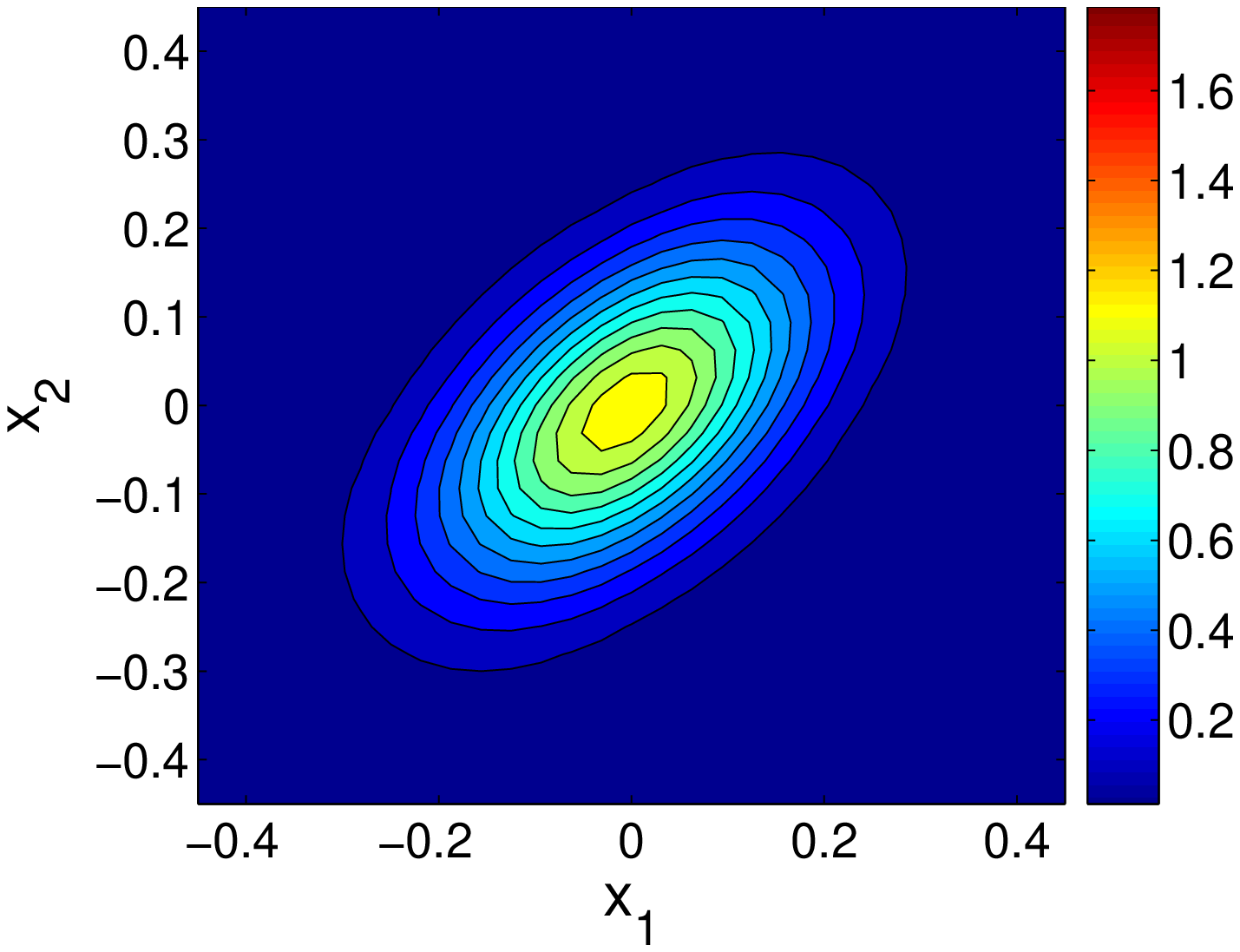}} &
  \resizebox{2.3in}{!}{\includegraphics{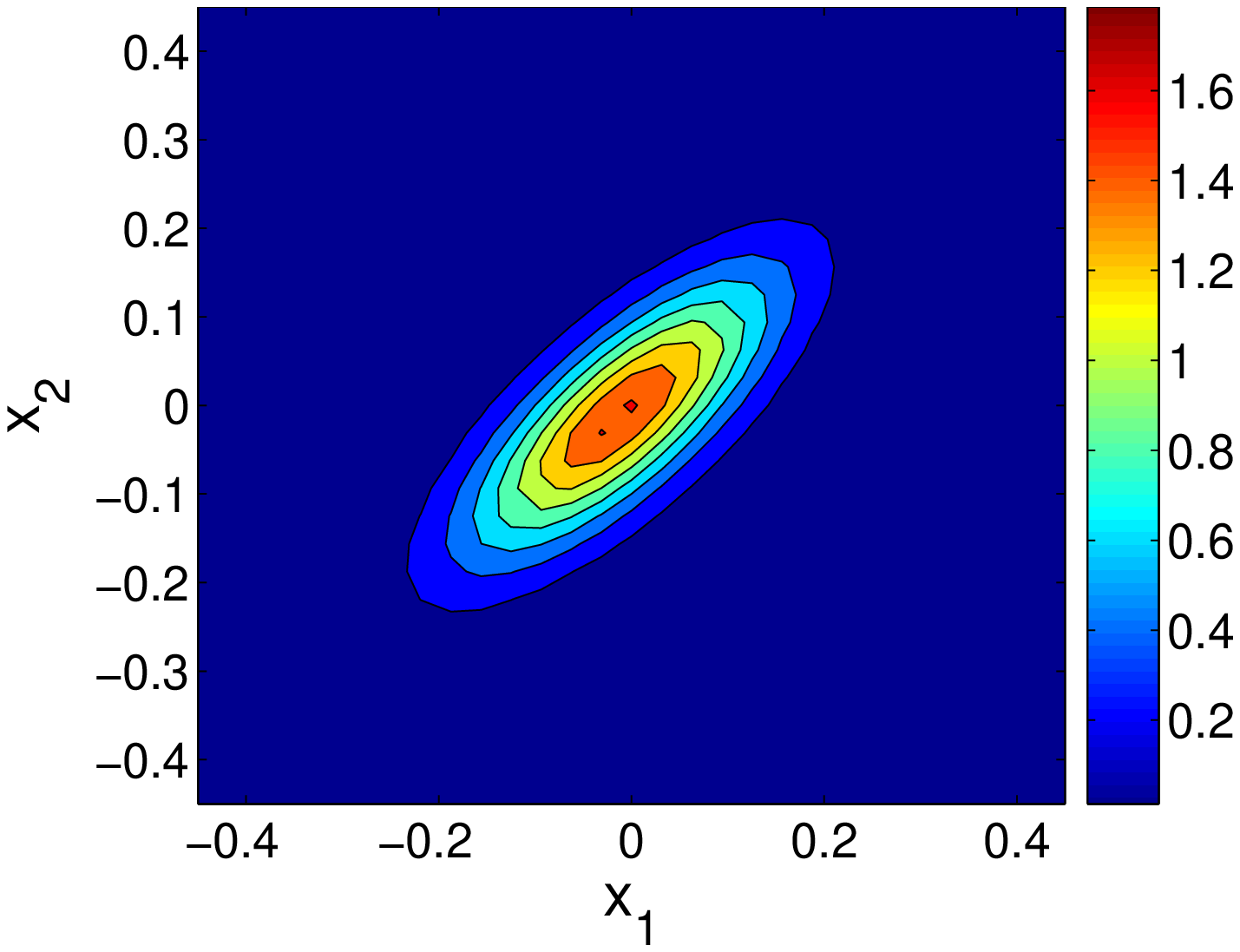}} \\
  \multicolumn{2}{c}{(b) True solution} \\[3mm]
  \resizebox{2.3in}{!}{\includegraphics{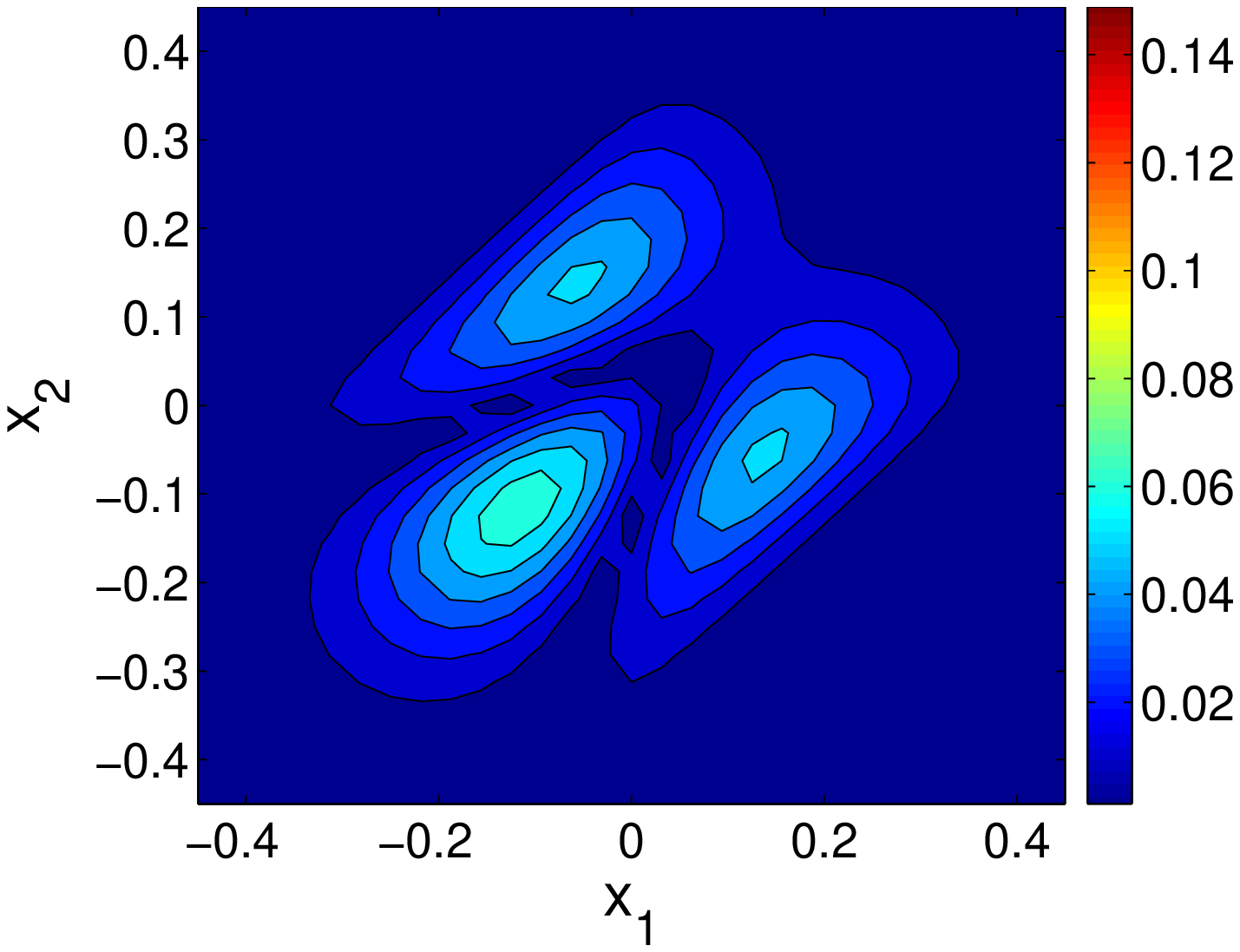}} &
  \resizebox{2.3in}{!}{\includegraphics{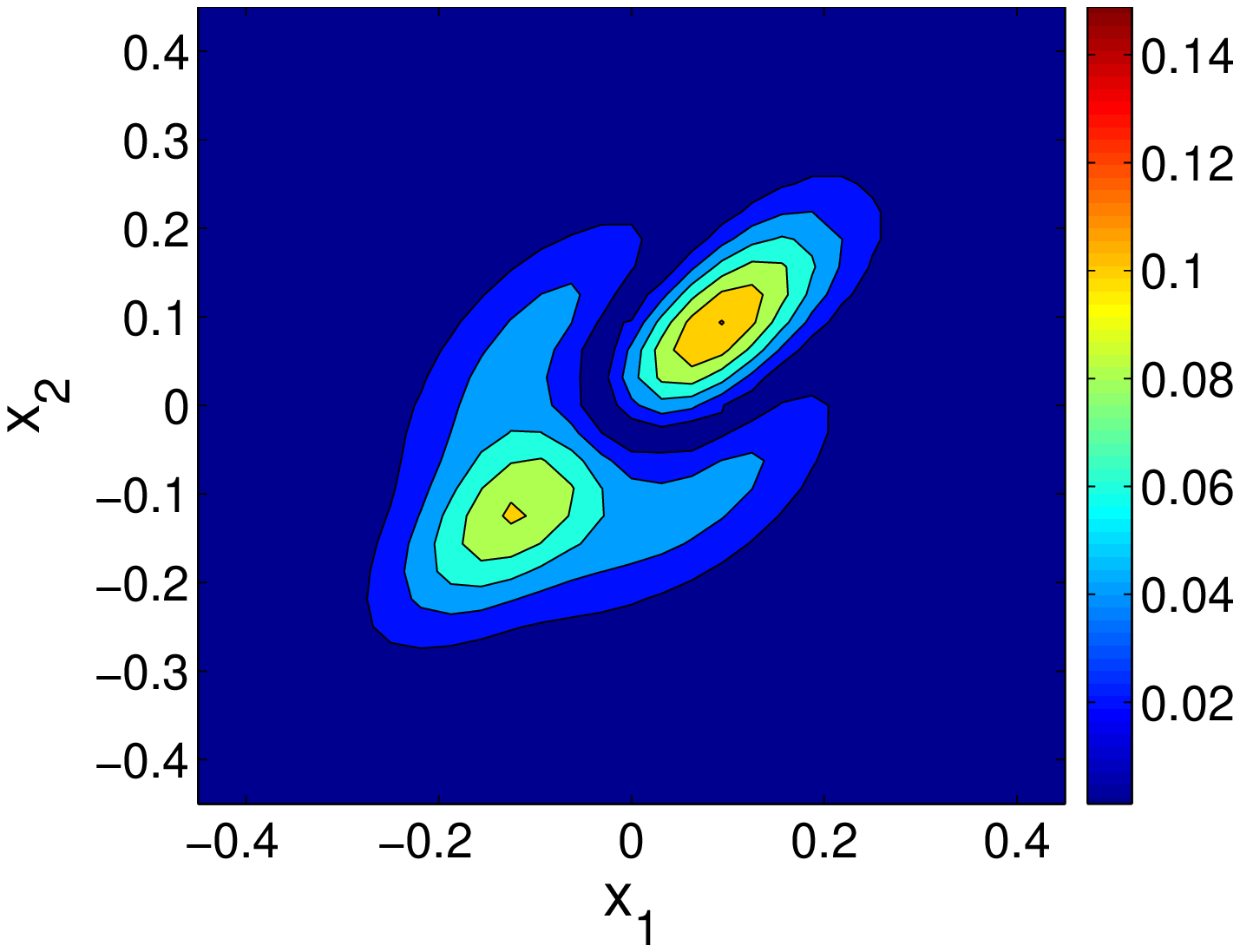}}\\
  \multicolumn{2}{c}{(c) Errors}
\end{tabular}
\caption{Example \ref{exa:3}, the comparison of the true solution and
  the solution by Eulerian Herman-Kluk propagator. Left: $T=0.5$;
  right: $T=1$.}\label{fig:ex3_amp}
\end{figure}

\section{Conclusion}\label{sec:conclusion}
We extend the formulation of frozen Gaussian approximation to general
linear strictly hyperbolic system. Based on the Eulerian formulation
of frozen Gaussian approximation, Eulerian methods are developed to
resolve the divergence problem of the Lagrangian method. Moreover, the
Eulerian methods can be also used for computing the Herman-Kluk propagator
of the Schr\"odinger equation in quantum mechanics. The performance of the
proposed methods is verified by numerical examples. This paper,
together with \cite{LuYang:CMS}, provides an efficient methodology for
computing high frequency wave propagation for general hyperbolic
systems with smooth coefficient.

\FloatBarrier

\bibliographystyle{amsalpha}
\bibliography{fga}

\end{document}